\documentclass[reqno, 11pt]{amsart}

\usepackage[english]{babel}
\usepackage{amsmath}
\usepackage{amsthm}
\usepackage{amsfonts}
\usepackage{amssymb}
\usepackage{anysize}
\usepackage{hyperref}
\usepackage{appendix}
\usepackage{mathtools}
\usepackage{graphicx}
\usepackage{xcolor}
\usepackage{amsmath,esint,bigints, amssymb, mathrsfs, amsthm,mathbbol,upgreek,exscale}

\newtheorem{defi}{Definition}[section]
\newtheorem{lem}[defi]{Lemma}
\newtheorem{theo}[defi]{Theorem}

\newtheorem{pro}[defi]{Proposition}

\newtheorem{rem}[defi]{Remark}

\DeclareMathOperator{\R}{\mathbb{R}}
\DeclareMathOperator{\C}{\mathbb{C}}
\DeclareMathOperator{\T}{\mathbb{T}}
\DeclareMathOperator{\D}{\mathbb{D}}

\allowdisplaybreaks 

\title[]{Vortex patches choreography for active scalar equations}

\author[C. Garc\'ia]{Claudia Garc\'ia}
\address{Departamento de Matem\'atica Aplicada and Excellence Research Unit ``Modeling Nature'' (MNat), Facultad de Ciencias. Universidad de Granada\\ 18071-Granada, Spain. \& IRMAR, Universit\'e de Rennes 1\\ Campus de
Beaulieu\\ 35~042 Rennes cedex\\ France}
\email{claudiagarcia@ugr.es}

{\thanks{This work has been partially supported by the MINECO-Feder (Spain) research grant number RTI2018--098850-B--I00, the Junta de Andaluc\'ia (Spain) project FQM 954 and the MECD (Spain) research grant FPU15/04094 (C.G.).}}

\begin{document}

\date{\today}

\begin{abstract}
This paper deals with the existence of $N$ vortex patches located at the vertex of a regular polygon with $N$ sides that rotate around the center of the polygon at a constant angular velocity. That is done for Euler and $\textnormal{(SQG)}_\beta$ equations, with $\beta\in(0,1)$, but may be also extended to more general models. The idea is the desingularization of the Thomsom polygon for the $N$ point vortex system, that is, $N$ point vortices located at the vertex of a regular polygon with $N$ sides. The proof is based on the study of the contour dynamics equation combined with the application of the infinite dimensional Implicit Function theorem and the well--chosen of the function spaces.
\end{abstract}

\maketitle

\section{Introduction}
This work deals with the dynamics of $N$ vortex patches located at the vertex of a  regular polygon with $N$ sides for Euler and generalized surface quasi--geostrophic equations. The motivation comes from the Thomson polygon for the $N$ point vortex problem which is a choreography for this system: the evolution of the points is a rotation of constant angular velocity.

  The generalized surface quasi--geostrophic equations describing the evolution of the potential temperature $\theta$ read as
\begin{eqnarray}	      \label{gSQG}     
       \left\{\begin{array}{ll}
          	\theta_t+(v\cdot \nabla) \theta=0, &\text{ in $[0,+\infty)\times\mathbb{R}^2$}, \\
         	 v=-\nabla^\perp (-\Delta)^{-1+\frac{\beta}{2}}\theta,&\text{ in $[0,+\infty)\times\mathbb{R}^2$}, \\
         	 \theta(0,x)=\theta_0(x),& \text{ with $x\in\mathbb{R}^2$}.
       \end{array}\right.
\end{eqnarray}
In this system $v$ refers to the velocity field, $\nabla^\perp=(-\partial_2,\partial_1)$ and $\beta\in[0,2)$. The velocity field is linked to the potential temperature $\theta$ via the operator $(-\Delta)^{-1+\frac{\beta}{2}}$ agreeing with
$$
(-\Delta)^{-1+\frac{\beta}{2}}\theta(x)=\int_{\R^2}K_{\beta}(x-y)\theta(y)dy,
$$
where
$$
K_{\beta}(x)=\left\{
\begin{array}{ll}
-\frac{1}{2\pi}\log|x|,&\beta=0,\\
\frac{C_\beta}{2\pi}\frac{1}{|x|^\beta},&\beta\in(0,2),
\end{array}
\right.
$$
and $C_\beta=\frac{\Gamma(\beta/2)}{2^{1-\beta}\Gamma\left(\frac{1-\beta}{2}\right)}$. 

The case $\beta=0$ corresponds to the Euler equations and $\beta=1$ to the surface quasi--geostrophic model. This model was proposed by C\'ordoba et al. in \cite{Cordoba} as an interpolation between Euler and surface quasi--geostrophic equations. See \cite{Constantin-Majda-Tabak} for its relation with the three dimensional incompressible Euler equations.

In this work we will be interested in the case $\beta\in[0,1)$ and we will distinguish between Euler equations and $\textnormal{(SQG)}_{\beta}$ for $\beta\in(0,1)$. Let us present here the Euler equations where $\omega$ stands for the vorticity

\begin{eqnarray}	      \label{Euler}     
       \left\{\begin{array}{ll}
          	\omega_t+(v\cdot \nabla) \omega=0, &\text{ in $[0,+\infty)\times\mathbb{R}^2$}, \\
         	 v=K*\omega,&\text{ in $[0,+\infty)\times\mathbb{R}^2$}, \\
         	 \omega(0,x)=\omega_0(x),& \text{ with $x\in\mathbb{R}^2$},
       \end{array}\right.
\end{eqnarray}
where $K(x)=\frac{1}{2\pi}\frac{x^\perp}{|x|^2}$.

Relative equilibria for these active scalar equations have been studied in the last few years. They survive in the class of Yudovich solutions \cite{Yudovich} (bounded and integrable functions), which are global for the Euler equations but only local for the $\textnormal{(SQG)}_{\beta}$ equations due to the singularity of the velocity field. For the case of vortex patches solutions, that is, $\theta_0(x)={\bf 1}_{D}$, for a bounded simply connected smooth domain $D$, there is a unique local solution in the patch form $\theta(t,x)={\bf 1}_{D_t}$. Furthermore, the boundary $\partial D_t$ is described by the contour dynamics equations
$$
\partial_t z_t(w)=\int_{\T}K_{\beta}(z_t(w)-z_t(\xi))z_t'(\xi)d\xi,
$$
where $z_t:\T\rightarrow \partial D_t$ is a lagrangian parametrization. In \cite{Chemin}, Chemin proved the global persistence of the boundary regularity in $\mathscr{C}^{1,\alpha}$, with $\alpha\in(0,1)$, for Euler equations, see also \cite{B-C, Serf}. However, a local persistence result is only known for $\textnormal{(SQG)}_{\beta}$ with $\beta\in(0,1)$, see \cite{CordobaCordoba, Gancedo, Kiselev, Rodrigo}.

We shall focus on relative equilibria or V--states, which consist in single or multiple patches moving without changing shape. The first known examples goes back to the stationary Rankine vortex (the circular patch) and Kirchhoff \cite{Kirchhoff} ellipses (which rotates at a constant angular velocity $\Omega=\frac{ab}{(a+b)^2}$, where $a$ and $b$ are the semi--axes of the ellipse). Later, Deem and Zabusky \cite{Deem-Zabusky} gave some numerical observations of the existence of more rotating patches with $m$--fold symmetry. Using bifurcation theory, Burbea \cite{Burbea} proved analytically the existence of these V--states close to the Rankine vortex. Following the approach of Burbea, there has been several works concerning the existence of V--states not only for Euler equations but also for $\textnormal{(SQG)}_{\beta}$ equations. We refer to \cite{Cas-Cor-Gom, Hassa-Hmi, HmidiMateuVerdera} for single rotating patches, \cite{DelaHoz-Hassainia-Hmidi, DelaHozHmidiMateuVerdera} for doubly--connected V--states, \cite{H-M} for corotating and counter--rotating vortex pairs, \cite{G-KVS} for the existence of K\'arm\'an Vortex Street structures and \cite{Cas0-Cor0-Gom, CastroCordobaGomezSerrano, C-C-GS-2,D-H-R, GHM, GHS, GPSY, HMW, Hmiditrivial, HmidiMateu} for other related works.

The $N$ point vortex system is found when the area of a patch shrink to zero. We will be interested in choreographies for the $N$ point vortex problem which are solutions where the $N$ vortices follow the same patch. Moreover, these choreographies can help later with the study of relative equilibria for active scalar equations as Euler or $\textnormal{(SQG)}_{\beta}$ equations. 

In \cite{Marchioro}, Marchioro and Pulvirenti proved the desingularization of point vortices for Euler equations. That is, they found smooth solutions $\omega_{\varepsilon}$ for Euler equations, suitably $\varepsilon$-concentrated around the points. We refer also to \cite{DavilaPinoMussoWei, Smets} for more information about this topic. Recently, Ao, D\'avila, Del Pino, Musso and Wei \cite{Ao} achieved also the desingularization of points for the $\textnormal{(SQG)}_{\beta}$ equations, with $\beta\in(0,1)$, for the special case of vortices traveling with constant speed along one axis or rotating with constant angular velocity. With different techniques, Hmidi and Mateu studied the desingularization of a pair of vortices by studying {\it small} vortex patches around each point. Note that their work deals with Euler and $\textnormal{(SQG)}_{\beta}$ equations, for $\beta\in(0,1)$. Following similar techniques, the author in \cite{G-KVS} found K\'arm\'an Vortex Street structures for different active scalar equations which include also Euler and $\textnormal{(SQG)}_{\beta}$.

Here we follow the approach developed in \cite{G-KVS, H-M} to study the desingularization of the $N$ point vortices located at the vertex of a regular polygon with $N\geq 2$ sides, that is,
\begin{equation}\label{intro-points}
z_m(0)=e^{\frac{i2\pi m}{N}}z_0(0),
\end{equation}
for any $m=0,\dots,N-1$ and where we take $z_0(0)=l\in\R$ located in the real axis. Assume that the evolution of the points is given by a general interaction  $G:\R\rightarrow\R$ smooth off zero, meaning
\begin{align}
z_m'(t)=&\sum_{m\neq k=0}^{N-1}\nabla^\perp_{z_m}G(|z_m(t)-z_k(t)|),\label{intro-evol-points-general}\\
z_m(0)=&e^{\frac{i2\pi m}{N}}z_0(0),\nonumber
\end{align}
for any $m=0,\dots,N-1$. Note that the classical Newtonian interaction stands for $G=\frac{1}{2\pi}\ln|\cdot|$ and $\textnormal{(SQG)}_{\beta}$ for $G=C_\beta\frac{1}{|\cdot|^\beta}$. We can check that the evolution of \eqref{intro-points} due to \eqref{intro-evol-points-general} is given by a rotation of constant angular velocity that we will denoted by $\Omega_0$, that is,
$$
z_m(t)=e^{i\Omega_0 t}z_m(0).
$$
The stability of the $N$ polygon (also called Thomson polygon) has been studied in the last years, we refer to \cite{Boatto, Calleja,KurakinYudovich}. They concluded that Thomson polygon is stable for $N\leq7$ but unstable for $N\geq 8$. 

Hence, the idea of this work is to find small enough vortex patches around \eqref{intro-points} evolving also as a rotation with some constant angular velocity. Our main result reads as follows.

\begin{theo}\label{intro-th}
Consider $l\in\R^\star$ and $N\geq 2$. Then, there exists $\varepsilon_0>0$ with the following property. For all $\varepsilon\in(0,\varepsilon_0)$, there is a simply--connected bounded domain $D^\varepsilon$, with center of masses $l$, such that
$$
\theta_0(x)=\frac{1}{\pi \varepsilon^2}\sum_{m=0}^{N-1}{\bf 1}_{e^{\frac{i 2\pi m}{N}}D^\varepsilon} (x),
$$
defines a rotating solution of \eqref{gSQG}, for $\beta\in[0,1)$, with some constant angular velocity $\Omega(\varepsilon)$. Moreover, $D^\varepsilon$ is at least $\mathscr{C}^1$.
\end{theo}
 
For the sake of simplicity, we will focus only on Euler and $\textnormal{(SQG)}_{\beta}$ equations, but same result happens also for more general models of the type
\begin{eqnarray}	      \label{general}     
       \left\{\begin{array}{ll}
          	q_t+(v\cdot \nabla) q=0, &\text{ in $[0,+\infty)\times\mathbb{R}^2$}, \\
         	 v=\nabla^\perp G* q,&\text{ in $[0,+\infty)\times\mathbb{R}^2$}, \\
         	 q(t=0,x)=q_0(x),& \text{ with $x\in\mathbb{R}^2$},
       \end{array}\right.
\end{eqnarray}
where $G$ is radial and smooth off zero. See \cite{G-KVS} for the desingularization of K\'arm\'an Vortex Street structures in the general system \eqref{general}.

In what follows, we briefly explain the idea behind Theorem \ref{intro-th} for Euler equations, and same strategy is later done for $\textnormal{(SQG)}_{\beta}$ equations. Our main task is to find simply connected bounded domains $D_m^\varepsilon$, for $m=0, \dots, N-1$, $N\geq 2$ and $\varepsilon>0$, such that the evolution of
\begin{equation}\label{intro-omega0}
\omega_{0,\varepsilon}(x)=\sum_{m=0}^{N-1}{\bf 1}_{D_m^\varepsilon}(x),
\end{equation}
is due to a rotation of constant angular velocity, that is,
\begin{equation}\label{intro-ansatz}
\omega(t,x)=\omega_{0,\varepsilon}(e^{-i\Omega t}x),
\end{equation}
for some $\Omega\in\R$. Moreover, we will consider that the center of masses of $D_0^\varepsilon$ is $l\in\R^\star$ and thus we take
$$
D_m^\varepsilon=e^{\frac{2\pi i m}{N}}D_0^\varepsilon.
$$
Note the following relation with the $N$ point vortex system. In the case that $D_0^\varepsilon=\varepsilon\D+l$, one finds for $\varepsilon\rightarrow 0$ the point vortex distribution \eqref{intro-points}.

Furthermore, inserting the ansatz \eqref{intro-ansatz} in the Euler equations, one finds the following equivalent equation
\begin{equation}\label{intro-eq}
\left(v_0(x)-\Omega x^\perp\right)\cdot n_{\partial D_0^\varepsilon}(x)=0, \quad x\in\partial D_0^\varepsilon,
\end{equation}
where $n_{\partial D_0^\varepsilon}$ stands for a unit normal vector to $\partial D_0^\varepsilon$. Moreover, we can parametrize $\partial D_0^\varepsilon-l$ by a conformal map $\Phi:\T\rightarrow \partial D_0^\varepsilon-l$ as
$$
\Phi(w)=\varepsilon(w+\varepsilon f(w)), \quad f(w)=\sum_{n\geq 1}a_n w^{-n}, \quad a_n\in\R, w\in\T,
$$
where $f$ does not depend on $\varepsilon$. Thus, after some computations \eqref{intro-eq} amounts to
$$
F(\Omega,\varepsilon,f)(w)=\textnormal{Re}\left[\left\{\overline{v_0(\Phi(w)+l)}+i\Omega(\overline{\Phi(w)}+l)\right\}w(w+\varepsilon f'(w))\right]=0, \quad w\in\T.
$$
The velocity field $v_0$ associated to \eqref{intro-omega0} takes the following expression
\begin{align*}
v_0(\Phi(w)+l)=&\frac{iw}{2\pi \varepsilon} +\frac{1}{4\pi^2}\overline{\int_{\T}\frac{\overline{w}-\overline{\xi}+\varepsilon(f(\overline{w})-f(\overline{\xi}))}{w-\xi+\varepsilon(f(w)-f(\xi))}f'(\xi)d\xi}\\
\nonumber&+\frac{1}{4\pi^2}\overline{\int_{\T}\frac{(w-\xi)(f(\overline{w})-f(\overline{\xi}))-(\overline{w}-\overline{\xi})(f(w)-f(\xi))}{(w-\xi)(w-\xi+\varepsilon(f(w)-f(\xi)))}d\xi}\\
\nonumber&+\sum_{m=1}^{N-1}\frac{1}{4\pi^2}\overline{\int_{\T}\frac{\overline{\xi}+\varepsilon f(\overline{\xi})}{\varepsilon e^{\frac{i2\pi m}{N}}({\xi}+\varepsilon f({\xi}))-\varepsilon ({w}+\varepsilon f({w}))+(e^{\frac{i2\pi m}{N}}-1)l}(1+\varepsilon f'(\xi))d\xi},
\end{align*}
for any $w\in\T$.

Then, the problem reduces to find non trivial roots of $F$ and thus get non trivial domains $D_m^\varepsilon$ such as \eqref{intro-omega0} evolves as \eqref{intro-ansatz}. Furthermore, we can prove that $F(\Omega_0,0,0)(w)=0$ for any $w\in\T$, where $\Omega_0$ is the angular velocity associated to the $N$ point vortices \eqref{intro-points}. 

The idea now is the use of the infinite dimensional Implicit Function theorem to $F$ in order to find solutions around the trivial one $(\Omega_0,0,0)$. Then, one needs that the linearized operator around such trivial solution is an isomorphism. If one formally compute such operator one gets
$$
\partial_f F(\Omega_0,0,0)h(w)=\frac{1}{2\pi}\textnormal{Im}[h'(w)].
$$
In order to have that $\partial_f F(\Omega_0,0,0)$ is an isomorphism, let us define the following function spaces
\begin{align}
X_{1+\alpha}=&\left\{f\in\mathscr{C}^{1+\alpha}(\T), \quad f(w)=\sum_{n\geq 1}a_n w^{-n}, \, a_n\in\R\right\},\label{intro-X}\\
Y_\alpha=&\left\{f\in\mathscr{C}^{\alpha}(\T), \quad f(e^{i\theta})=\sum_{n\geq 2}a_n \sin(n\theta),\, a_n\in\R\right\},\label{intro-Y}
\end{align}
for $\alpha\in(0,1)$. Thus, it is clear that $\partial_f F(\Omega_0,0,0):X_{1+\alpha}\rightarrow Y_\alpha$ is an isomorphism, for any $\alpha\in(0,1)$.

However, the nonlinear functional $F$ is not a priori well--defined in such function spaces. Indeed, one can prove that it belongs to $\mathscr{C}^{\alpha}(\T)$ and takes the form
$$
F(\Omega,\varepsilon,f)(e^{i\theta})=\sum_{n\geq 1}f_n\sin(n\theta),
$$
with $\theta\in[0,2\pi)$ and $f_n\in\R$. Hence, in order to have that $F\in Y_\alpha$ one needs an extra condition: $f_1=0$. To achieve that, we assume that $\Omega$ depends appropriately on $(\varepsilon,f)$. Indeed, by assuming
\begin{equation}\label{intro-Omega-exp}
\Omega(\varepsilon,f)=\frac{i\int_{\T}J(\varepsilon,f)(w)(w-\overline{w})(1+\varepsilon f'(w))dw}{\int_{\T}(1+\varepsilon f'(w))(w-\overline{w})(l+\varepsilon \overline{w}+\varepsilon^2 f(\overline{w}))dw}.
\end{equation}
one has now that
$$
F(\Omega(\varepsilon,f),\varepsilon,f)(e^{i\theta})=\sum_{n\geq 2}f_n\sin(n\theta),
$$
and then it is worthy to work with $\tilde{F}$ defined as
$$
\tilde{F}(\varepsilon,f)=F(\Omega(\varepsilon,f),\varepsilon,f),
$$
instead of $F$. Moreover, one can check that $\Omega(0,f)=\Omega_0$, for any $f\in X_{1+\alpha}$, and then
$$
\partial_f \tilde{F}(0,0)=(\partial_f F)(\Omega(0,0),0,0)+(\partial_\Omega F)(\Omega(0,0),0,0)(\partial_f \Omega)(0,0)=\partial_f F(\Omega_0,0,0),
$$
which is an isomorphism in those spaces as it was mentioned before. Thus the function spaces \eqref{intro-X}--\eqref{intro-Y} seem to be well chosen in this case.
Similar techniques but with more involved computations work for $\textnormal{(SQG)}_{\beta}$, for $\beta\in(0,1)$.

This work is organized as follows. In Section \ref{Sec-points} we describe the $N$ vortex model studying the evolution of the $N$ polygon. Section \ref{Sec-Euler} refers to the desingularization of the $N$ polygon for Euler equations whereas Section \ref{Sec-gsqg} stands for the one in the $\textnormal{(SQG)}_{\beta}$ equations.

\section{Point vortex model}\label{Sec-points}
This section aims to study the dynamics of $N$ point vortices located at the vertex of a regular polygon with $N$ sides. We shall notice that their evolution is given by a rotation of constant angular velocity. That occurs not only for the Newtonian interaction but also for more general interactions.

Then, consider initially a regular polygon with $N$ sides and $N$ point vortices located at the vertex of the polygon. Furthermore, assume that the polygon has its center at the origin and that there is a vertex in the horizontal axis, meaning $z_0(0)=l\in\R$, then the others vertex are described by
$$
z_m(0)=e^{\frac{i2\pi m}{N}}z_0(0),
$$
for any $m=0,\dots,N-1$. The evolution of such points with Newtonian interaction is given by the classical $N$--vortex problem, that is,
\begin{align}
z_m'(t)=&\frac{1}{2\pi}\sum_{m\neq k=0}^{N-1}\frac{(z_m(t)-z_k(t))^\perp}{|z_m(t)-z_k(t)|^2},\label{evol-points}\\
z_m(0)=&e^{\frac{i2\pi m}{N}}z_0(0),\nonumber
\end{align}
for any $m=0,\dots,N-1$. 

On the other hand, assuming that the interaction between the points is a general function $G:\R\rightarrow\R$ smooth off zero, one has
\begin{align}
z_m'(t)=&\sum_{m\neq k=0}^{N-1}\nabla^\perp_{z_m}G(|z_m(t)-z_k(t)|),\label{evol-points-general}\\
z_m(0)=&e^{\frac{i2\pi m}{N}}z_0(0),\nonumber
\end{align}
for any $m=0,\dots,N-1$.  See \cite{Newton} for more details.

 In the following proposition, we show that the evolution of such points is given through a rotation of constant angular velocity. For a better understanding, we first give the result for the Newtonian interaction: 
\begin{pro}\label{Prop-points}
Let $z_m(0)=e^{\frac{i2\pi m}{N}}z_0(0)$ and $z_0(0)=l\in\R$, for any $m=0, \dots, N-1$. Then, $z_m(t)=e^{i\Omega t}z_m(0)$, where 
\begin{equation}\label{omega}
\Omega=\frac{1}{2\pi l^2}\sum_{k=1}^{N-1}\frac{1}{1-e^{\frac{i2\pi k}{N}}},
\end{equation} 
for any $m=0,\dots,N-1$.
\end{pro}
\begin{rem}
It is well--known in the literature that in the case of a vortex pair with strength $1$ and separated by a distance $d$, one has that it rotates at angular velocity $\Omega=\frac{1}{\pi d^2}$. Note that it agrees with \eqref{omega} by taking $N=2$ and $d=2l$.
\end{rem}
\begin{proof}
From \eqref{evol-points} we easily get that
$$
\frac{d}{dt}\sum_{m=0}^{N-1}z_m(t)=0,
$$
and then
$$
\sum_{m=0}^{N-1}z_m(t)=\sum_{m=0}^{N-1}z_m(0).
$$
Note that from the particular initial distribution of points, one has that
$$
\sum_{m=0}^{N-1}z_m(0)=0,
$$
implying 
$$
\sum_{m=0}^{N-1}z_m(t)=\sum_{m=0}^{N-1}z_m(0)=0.
$$
Insert now the ansatz $z_m(t)=e^{i\Omega t}z_m(0)$ in \eqref{evol-points}, thus we achieve
\begin{equation}\label{eq-points1}
i\Omega e^{i\Omega t}z_m(0)=\frac{i}{2\pi}\sum_{m\neq k=0}^{N-1}\frac{e^{i\Omega t}(z_m(0)-z_k(0))}{|z_m(0)-z_k(0)|^2},
\end{equation}
which agrees with
$$
\Omega z_m(0)=\frac{1}{2\pi}\sum_{m\neq k=0}^{N-1}\frac{z_m(0)-z_k(0)}{|z_m(0)-z_k(0)|^2}.
$$
From the definition of $z_m(0)$ in \eqref{omega}, one has that 
$$
\Omega e^{\frac{i2\pi m}{N}}z_0(0)=\frac{1}{2\pi}\sum_{m\neq k=0}^{N-1}\frac{e^{\frac{i2\pi m}{N}}z_0(0)-e^{\frac{i2\pi k}{N}}z_0(0)}{|e^{\frac{i2\pi m}{N}}z_0(0)-e^{\frac{i2\pi k}{N}}z_0(0)|^2}.
$$
Multiplying everything by $e^{-\frac{i2\pi m}{N}}$ and making the change of variables $k-m\mapsto k$ one finds
$$
\Omega=\frac{1}{2\pi|z_0(0)|^2}\sum_{0\neq k=-m}^{N-1-m}\frac{1-e^{\frac{i2\pi k}{N}}}{|1-e^{\frac{i2\pi k}{N}}|^2}.
$$
Note that the above expression equals to
\begin{equation}\label{omega1}
\Omega=\frac{1}{2\pi|z_0(0)|^2}\sum_{k=1}^{N-1}\frac{1-e^{\frac{i2\pi k}{N}}}{|1-e^{\frac{i2\pi k}{N}}|^2}=\frac{1}{2\pi|z_0(0)|^2}\sum_{k=1}^{N-1}\frac{1}{1-e^{\frac{i2\pi k}{N}}},
\end{equation}
that does not depend on $m$. Hence, there exists a constant $\Omega$ verifying \eqref{eq-points1} for any $m=1,\dots,N-1$. Moreover, from \eqref{omega1} we can check that $\Omega\in\R$, and using that $|z_0(0)|=l$ we achieve \eqref{omega}.
\end{proof}

The previous result can easily be generalized to the system \eqref{evol-points-general}.

\begin{pro}\label{Prop-points-general}
Let $G:\R\rightarrow \R$ be smooth off zero, $z_m(0)=e^{\frac{i2\pi m}{N}}z_0(0)$ and $z_0(0)=l\in\R$, for any $m=0, \dots, N-1$. Then, $z_m(t)=e^{i\Omega t}z_m(0)$, where 
\begin{equation}\label{omega-general}
\Omega=\frac{1}{l}\sum_{k=1}^{N-1}\frac{1-e^{\frac{i2\pi k}{N}}}{|1-e^{\frac{i2\pi k}{N}}|}G'(l|1-e^{\frac{i2\pi k}{N}}|),
\end{equation} 
for any $m=0,\dots,N-1$.
\end{pro}
\begin{proof}
This proof is very similar to the previous one. Indeed, assume that we have a rotation evolution for \eqref{evol-points-general}, that is, 
$$
z_m(t)=e^{i\Omega t}z_m(0).
$$
Hence, one finds that $\Omega$ must satisfy

\begin{equation}\label{eq-points1-g}
i\Omega e^{i\Omega t}z_m(0)=i\sum_{m\neq k=0}^{N-1}\frac{e^{i\Omega t}(z_m(0)-z_k(0))}{|z_m(0)-z_k(0)|}G'(|z_m(0)-z_k(0)|),
\end{equation}
which can be reduced to
$$
\Omega z_m(0)=\sum_{m\neq k=0}^{N-1}\frac{z_m(0)-z_k(0)}{|z_m(0)-z_k(0)|}G'(|z_m(0)-z_k(0)|).
$$
Moreover, the location of $z_m(0)$ amounts to 
$$
\Omega e^{\frac{i2\pi m}{N}}z_0(0)=\frac{1}{2\pi}\sum_{m\neq k=0}^{N-1}\frac{e^{\frac{i2\pi m}{N}}z_0(0)-e^{\frac{i2\pi k}{N}}z_0(0)}{|e^{\frac{i2\pi m}{N}}z_0(0)-e^{\frac{i2\pi k}{N}}z_0(0)|}G'(|e^{\frac{i2\pi m}{N}}z_0(0)-e^{\frac{i2\pi k}{N}}z_0(0)|),
$$
and again multiplying everything by $e^{-\frac{i2\pi m}{N}}$ and making the change of variables $k-m\mapsto k$ it yields
$$
\Omega=\frac{1}{|z_0(0)|}\sum_{0\neq k=-m}^{N-1-m}\frac{1-e^{\frac{i2\pi k}{N}}}{|1-e^{\frac{i2\pi k}{N}}|}G'(|z_0(0)||1-e^{\frac{i2\pi k}{N}}|).
$$
Thus, it can be written as
\begin{equation}\label{omega1-g}
\Omega=\frac{1}{|z_0(0)|}\sum_{k=1}^{N-1}\frac{1-e^{\frac{i2\pi k}{N}}}{|1-e^{\frac{i2\pi k}{N}}|}G'(|z_0(0)||1-e^{\frac{i2\pi k}{N}}|),
\end{equation}
that does not depend on $m$ implying the existence of $\Omega$ verifying \eqref{eq-points1} for any $m=1,\dots,N-1$. Using now that $|z_0(0)|=l$ we achieve 
\begin{equation}\label{omega2-g}
\Omega=\frac{1}{l}\sum_{k=1}^{N-1}\frac{1-e^{\frac{i2\pi k}{N}}}{|1-e^{\frac{i2\pi k}{N}}|}G'(l|1-e^{\frac{i2\pi k}{N}}|).
\end{equation}
\end{proof}

\begin{rem}\label{rem-sqg}
Note that in the case of the generalized quasi--geostrophic equation, one has
$$
\Omega=\frac{\beta C_\beta}{l^{2+\beta}}\sum_{k=1}^{N-1}\frac{1-e^{\frac{i2\pi k}{N}}}{|1-e^{\frac{i2\pi k}{N}}|^{2+\beta}}.
$$
\end{rem}

\section{Vortex patch polygon for Euler equations}\label{Sec-Euler}

The purpose of this section is to study rotating vortex patches distributed in a regular polygon, which is motivated by Proposition \ref{Prop-points}. The idea is to desingularize the point vortices of Proposition \ref{Prop-points} by finding {\it small} vortex patches around them. First, we shall write down the equation that will characterize this type of solutions and later find the suitable function spaces where the problem is well--posed. Secondly, the implementation of the infinite dimensional Implicit Function spaces will lead to the existence of such {\it small} patches around the point vortices.

More specifically, our main task in this section is to find domains $D_m^\varepsilon$, for $m=0,\dots, N-1$, $N\geq 2$ and $\varepsilon>0$, such that the initial data
\begin{equation}\label{initial-data}
\omega_{0,\varepsilon}(x)=\frac{1}{\pi \varepsilon^2}\sum_{m=0}^{N-1}{\bf 1}_{ D^\varepsilon_m}(x),
\end{equation}
evolves as a rotation of constant angular velocity. That is, there exists $\Omega\in\R$ and $\varepsilon>0$ such that the evolution of \eqref{initial-data} is given by
\begin{equation}\label{initial-data-evol}
\omega(t,x)=\omega_{0,\varepsilon}(e^{-i\Omega t}x).
\end{equation}
Indeed, we will consider that $D^\varepsilon_m$ are located in the plane as the point vortices in Proposition \ref{Prop-points}. That is, take $D_0^\varepsilon$ with center of masses $l\in\R^\star$ and
\begin{equation}\label{Dm}
D_m^\varepsilon=e^{i\frac{2\pi m}{N}}D_0^\varepsilon,
\end{equation}
for $m\geq 1$. Note that assuming that $D_0^\varepsilon=\varepsilon\D+l$,  one finds for $\varepsilon\rightarrow 0$ in \eqref{initial-data} the point vortex distribution of Proposition \ref{Prop-points}:
\begin{equation}\label{initial-data0}
\omega_{0,0}(x)=\sum_{m=0}^{N-1}{\delta}_{e^{i\frac{2\pi m}{N}}l}(x).
\end{equation}
Now assume that the evolution of \eqref{initial-data} is given by \eqref{initial-data-evol}. In that case, the Euler equations agree with
\begin{equation}\label{eq-1}
\left(v_0(x)-\Omega x^\perp\right)\cdot n_{\partial D^\varepsilon_m}(x)=0, \quad x\in\partial  D_m^\varepsilon,
\end{equation}
for any $m=0,\dots N-1$. Here $n_{\partial  D_m^\varepsilon}$ stands for a unit normal vector to $\partial D_m^\varepsilon$.  Indeed using \eqref{Dm}, one finds that \eqref{eq-1} is equivalent to check such equation just for $x\in\partial D_0^\varepsilon$, that is,
\begin{equation}\label{eq-2}
\left(v_0(x)-\Omega x^\perp\right)\cdot n_{\partial D_0^\varepsilon}(x)=0, \quad x\in\partial D_0^\varepsilon.
\end{equation}

The next step is to assume some symmetry for $D_0^\varepsilon$ in order to simplify \eqref{eq-1}. Motivated by the works \cite{H-M} and \cite{G-KVS} where the authors desingularized different point vortex patterns, take that $D_0^\varepsilon-l$ is parametrized by a conformal map from $\T$ into $\partial D_0^\varepsilon-l$ in the sense
\begin{equation}\label{conformal-map}
\Phi(w)=\varepsilon \left(w+\varepsilon f(w)\right), \quad f(w)=\sum_{n\geq 1}a_n w^{-n},\quad a_n\in\R, w\in\T,
\end{equation}
where $f$ does not depend on $\varepsilon$. 

Using such parametrization, straightforward computations imply that \eqref{eq-2} agrees with
\begin{equation}\label{eq-3}
\textnormal{Re}\left[\left\{\overline{v_0(\Phi(w)+l)}+i\Omega(\overline{\Phi(w)}+l)\right\}w\Phi'(w)\right]=0, \quad w\in\T.
\end{equation}
In the next lemma we shall give an useful expression for $v_0(\Phi(w)+l)$:
\begin{lem}\label{Lem-velocity}
The velocity field $v_0$ associated to the initial data \eqref{initial-data}, assuming \eqref{conformal-map}, agrees with
\begin{align}\label{velocity-exp}
\nonumber&v_0(\Phi(w)+l)=\frac{iw}{2\pi \varepsilon} +\frac{1}{4\pi^2}\overline{\int_{\T}\frac{\overline{w}-\overline{\xi}+\varepsilon(f(\overline{w})-f(\overline{\xi}))}{w-\xi+\varepsilon(f(w)-f(\xi))}f'(\xi)d\xi}\\
\nonumber&+\frac{1}{4\pi^2}\overline{\int_{\T}\frac{(w-\xi)(f(\overline{w})-f(\overline{\xi}))-(\overline{w}-\overline{\xi})(f(w)-f(\xi))}{(w-\xi)(w-\xi+\varepsilon(f(w)-f(\xi)))}d\xi}\\
\nonumber&+\sum_{m=1}^{N-1}\frac{1}{4\pi^2}\overline{\int_{\T}\frac{\overline{\xi}+\varepsilon f(\overline{\xi})}{\varepsilon e^{\frac{i2\pi m}{N}}({\xi}+\varepsilon f({\xi}))-\varepsilon ({w}+\varepsilon f({w}))+(e^{\frac{i2\pi m}{N}}-1)l}(1+\varepsilon f'(\xi))d\xi}\\
=:&\frac{iw}{2\pi \varepsilon} +\overline{J(\varepsilon,f)(w)},
\end{align}
for any $w\in\T$.
\end{lem}
\begin{proof}
Note that the velocity field associated to \eqref{initial-data} is given by
\begin{equation*}
v_0(x)=\sum_{m=0}^{N-1}\frac{i}{2\pi^2\varepsilon^2}\int_{D_m^\varepsilon}\frac{dA(y)}{\overline{x}-\overline{y}},
\end{equation*}
for $x\in \partial D_0^\varepsilon$. By using Stokes theorem, it agrees with
\begin{equation}\label{v0-eq1}
v_0(x)=\sum_{m=0}^{N-1}\frac{1}{4\pi^2\varepsilon^2}\overline{\int_{\partial D_m^\varepsilon}\frac{\overline{x}-\overline{\xi}}{x-\xi}d\xi},
\end{equation}
for $x\in \partial D_0^\varepsilon$. Note now that if $m\neq 0$, one automatically has the following

\begin{equation*}
\int_{\partial D_m^\varepsilon}\frac{d\xi}{x-\xi}=0,
\end{equation*}
since $x\in D_0^\varepsilon$. Hence \eqref{v0-eq1} can be written as
\begin{equation*}
v_0(x)=\frac{1}{4\pi^2\varepsilon^2}\overline{\int_{\partial D_0^\varepsilon}\frac{\overline{x}-\overline{\xi}}{x-\xi}d\xi}+\sum_{m=1}^{N-1}\frac{1}{4\pi^2\varepsilon^2}\overline{\int_{\partial D_m^\varepsilon}\frac{\overline{\xi}}{\xi-x}d\xi},
\end{equation*}

Thanks to \eqref{Dm}, it amounts to
\begin{equation*}
v_0(x)=\frac{1}{4\pi^2\varepsilon^2}\overline{\int_{\partial D_0^\varepsilon}\frac{\overline{x}-\overline{\xi}}{x-\xi}d\xi}+\sum_{m=1}^{N-1}\frac{1}{4\pi^2\varepsilon^2}\overline{\int_{\partial D_0^\varepsilon}\frac{\overline{\xi}}{e^{\frac{i2\pi m}{N}}\xi-x}d\xi},
\end{equation*}
for $x\in \partial D_0^\varepsilon$. Making now the change of variables $x\mapsto x'+l$, where $x'\in\Phi(\T)$, the previous expression is equivalent to
\begin{equation*}
v_0(x+l)=\frac{1}{4\pi^2\varepsilon^2}\overline{\int_{\partial \Phi(\T)}\frac{\overline{x}-\overline{\xi}}{x-\xi}d\xi}+\sum_{m=1}^{N-1}\frac{1}{4\pi^2\varepsilon^2}\overline{\int_{\partial \Phi(\T)}\frac{\overline{\xi}}{e^{\frac{i2\pi m}{N}}\xi-x+(e^{\frac{i2\pi m}{N}}-1)l}d\xi}.
\end{equation*}

Introducing the conformal map \eqref{conformal-map} we get
\begin{align*}
&v_0(\Phi(w)+l)=\frac{1}{4\pi^2\varepsilon}\overline{\int_{\T}\frac{\overline{w}-\overline{\xi}+\varepsilon(f(\overline{w})-f(\overline{\xi}))}{w-\xi+\varepsilon(f(w)-f(\xi))}(1+\varepsilon f'(\xi))d\xi}\\
&+\sum_{m=1}^{N-1}\frac{1}{4\pi^2}\overline{\int_{\T}\frac{\overline{\xi}+\varepsilon f(\overline{\xi})}{\varepsilon e^{\frac{i2\pi m}{N}}({\xi}+\varepsilon f({\xi}))-\varepsilon ({w}+\varepsilon f({w}))+(e^{\frac{i2\pi m}{N}}-1)l}(1+\varepsilon f'(\xi))d\xi}\\
&=:\overline{I_1^\varepsilon(w)+I_2^\varepsilon(w)},
\end{align*}
for any $w\in\T$.  By virtue of the Residue theorem and following the ideas of \cite[Page 709]{H-M}, we get that
\begin{align*}
I_1^\varepsilon(w)=&-\frac{i}{2\pi \varepsilon}\overline{w}+\frac{1}{4\pi^2}\int_{\T}\frac{\overline{w}-\overline{\xi}+\varepsilon(f(\overline{w})-f(\overline{\xi}))}{w-\xi+\varepsilon(f(w)-f(\xi))}f'(\xi)d\xi\\
&+\frac{1}{4\pi^2}\int_{\T}\frac{(w-\xi)(f(\overline{w})-f(\overline{\xi}))-(\overline{w}-\overline{\xi})(f(w)-f(\xi))}{(w-\xi)(w-\xi+\varepsilon(f(w)-f(\xi)))}d\xi.
\end{align*}
Finally, we arrive at
\begin{align*}
&v_0(\Phi(w)+l)=\frac{iw}{2\pi \varepsilon} +\frac{1}{4\pi^2}\overline{\int_{\T}\frac{\overline{w}-\overline{\xi}+\varepsilon(f(\overline{w})-f(\overline{\xi}))}{w-\xi+\varepsilon(f(w)-f(\xi))}f'(\xi)d\xi}\\
&+\frac{1}{4\pi^2}\overline{\int_{\T}\frac{(w-\xi)(f(\overline{w})-f(\overline{\xi}))-(\overline{w}-\overline{\xi})(f(w)-f(\xi))}{(w-\xi)(w-\xi+\varepsilon(f(w)-f(\xi)))}d\xi}\\
&+\sum_{m=1}^{N-1}\frac{1}{4\pi^2}\overline{\int_{\T}\frac{\overline{\xi}+\varepsilon f(\overline{\xi})}{\varepsilon e^{\frac{i2\pi m}{N}}({\xi}+\varepsilon f({\xi}))-\varepsilon ({w}+\varepsilon f({w}))+(e^{\frac{i2\pi m}{N}}-1)l}(1+\varepsilon f'(\xi))d\xi}\\
=:&\frac{iw}{2\pi \varepsilon} +\overline{J(\varepsilon,f)(w)},
\end{align*}
for any $w\in\T$.
\end{proof}

By virtue of the above Lemma \ref{Lem-velocity}, we have that \eqref{eq-3} is equivalent to
\begin{equation}\label{F}
F(\varepsilon, \Omega, f)(w):=\textnormal{Re}\left[\left\{I(\varepsilon,f)(w)+i\Omega\varepsilon(\overline{w}+\varepsilon f(\overline{w}))+i\Omega l\right\}w(1+\varepsilon f'(w))\right]=0,
\end{equation}
for any $w\in\T$. The function $I(\varepsilon,f)$ is connected to the velocity field and is given by
\begin{align}\label{I-exp}
&I(\varepsilon,f)(w)=-\frac{i\overline{w}}{2\pi\varepsilon}+J(\varepsilon,f)(w),
\end{align}
where $J(\varepsilon,f)$ is defined in \eqref{velocity-exp}, and has the following expression
\begin{align}\label{J-exp}
J(\varepsilon,f)(w)=&\frac{1}{4\pi^2}\int_{\T}\frac{\overline{w}-\overline{\xi}+\varepsilon(f(\overline{w})-f(\overline{\xi}))}{w-\xi+\varepsilon(f(w)-f(\xi))}f'(\xi)d\xi\\
\nonumber&+\frac{1}{4\pi^2}\int_{\T}\frac{(w-\xi)(f(\overline{w})-f(\overline{\xi}))-(\overline{w}-\overline{\xi})(f(w)-f(\xi))}{(w-\xi)(w-\xi+\varepsilon(f(w)-f(\xi)))}d\xi\\
\nonumber&+\frac{1}{4\pi^2}\sum_{m=1}^{N-1}\int_{\T}\frac{\overline{\xi}+\varepsilon f(\overline{\xi})}{\varepsilon e^{\frac{i2\pi m}{N}}({\xi}+\varepsilon f({\xi}))-\varepsilon ({w}+\varepsilon f({w}))+(e^{\frac{i2\pi m}{N}}-1)l}(1+\varepsilon f'(\xi))d\xi,
\end{align}
for $w\in\T$. By using $J(\varepsilon,f)$, one finds that \eqref{F} agrees with
\begin{equation}\label{F-2}
F(\varepsilon, \Omega, f)(w):=\textnormal{Re}\left[-\frac{i}{2\pi}f'(w)+\left\{J(\varepsilon,f)(w)+i\Omega\varepsilon(\overline{w}+\varepsilon f(\overline{w}))+i\Omega l\right\}w(1+\varepsilon f'(w))\right]=0,
\end{equation}
for any $w\in\T$.

In order to find solutions as \eqref{initial-data}--\eqref{initial-data-evol}, one has to study the roots of $F$. First, let us show in the next lemma some symmetry properties of $F(\varepsilon,\Omega,f)$:

\begin{lem}\label{Lem-F-sym}
If $(\Omega,\varepsilon)\in\R^2$ and $f\in \mathscr{C}^1$ satisfies 
\begin{equation}\label{f-exp}
f(w)=\sum_{n\geq 1}a_nw^{-n},\quad a_n\in\R,w\in\T,
\end{equation}
hence 
\begin{equation}\label{F-sym2}
F(\Omega,\varepsilon,f)(\overline{w})=-F(\Omega,\varepsilon,f)(w), \quad w\in\T.
\end{equation} 
That implies 
\begin{equation}\label{F-sym}
F(\Omega,\varepsilon,f)(e^{i\theta})=\sum_{n\geq 1}f_n \sin(\theta), \quad f_n\in\R, \theta\in[0,2\pi].
\end{equation}
\end{lem}
\begin{proof}
Note that if $f$ satisfies \eqref{f-exp}, hence $f(\overline{w})=\overline{f(w)}$ and $f'(\overline{w})=\overline{f(w)}$. Moreover, notice that the first part of \eqref{F-2} satisfies \eqref{F-sym2}:
\begin{align*}
\textnormal{Re}\left[-\frac{i}{2\pi}f'(\overline{w})\right]=-\textnormal{Re}\left[\overline{-\frac{i}{2\pi}f'(w)}\right]=-\textnormal{Re}\left[{-\frac{i}{2\pi}f'(w)}\right].
\end{align*}
The last part of \eqref{F-2} trivially satisfies \eqref{F-sym2}, that is,
\begin{align*}
\textnormal{Re}\left[\left\{i\Omega\varepsilon(\overline{w}+\varepsilon f(\overline{w}))+i\Omega l\right\}\overline{w}(1+\varepsilon f'(\overline{w}))\right]=&-\textnormal{Re}\left[\overline{\left\{i\Omega\varepsilon(\overline{w}+\varepsilon f(\overline{w}))+i\Omega l\right\}{w}(1+\varepsilon f'({w}))}\right]\\
=&-\textnormal{Re}\left[{\left\{i\Omega\varepsilon(\overline{w}+\varepsilon f(\overline{w}))+i\Omega l\right\}{w}(1+\varepsilon f'({w}))}\right].
\end{align*}
Hence, in order to check \eqref{F-sym2} it remains to check that
\begin{equation}\label{J-sym}
J(\varepsilon,f)(\overline{w})=-\overline{J(\varepsilon,f)(w)}.
\end{equation}
Recall the following property of the complex integrals over $\T$:
\begin{equation}\label{property-sym-int}
\overline{\int_{\T}f(\xi)d\xi}=-\int_{\T}\overline{f(\overline{\xi})}d\xi,
\end{equation}
for a complex function $f$. Then,
\begin{align*}
-\overline{J(\varepsilon,f)(w)}=&\frac{1}{4\pi^2}\int_{\T}\frac{{w}-\overline{\xi}+\varepsilon(f({w})-f(\overline{\xi}))}{\overline{w}-\xi+\varepsilon(f(\overline{w})-f(\xi))}f'(\xi)d\xi\\
\nonumber&+\frac{1}{4\pi^2}\int_{\T}\frac{(\overline{w}-\xi)(f({w})-f(\overline{\xi}))-({w}-\overline{\xi})(f(\overline{w})-f(\xi))}{(\overline{w}-\xi)(\overline{w}-\xi+\varepsilon(f(\overline{w})-f(\xi)))}d\xi\\
\nonumber&+\sum_{m=1}^{N-1}\int_{\T}\frac{\overline{\xi}+\varepsilon f(\overline{\xi})}{\varepsilon e^{-\frac{i2\pi m}{N}}({\xi}+\varepsilon f({\xi}))-\varepsilon ({\overline{w}}+\varepsilon f({\overline{w}}))+(e^{-\frac{i2\pi m}{N}}-1)l}(1+\varepsilon f'(\xi))d\xi.
\end{align*}
Note that we can do a change of variables in the sum $m\mapsto -m$ getting 
\begin{align*}
-\overline{J(\varepsilon,f)(w)}=&\frac{1}{4\pi^2}\int_{\T}\frac{{w}-\overline{\xi}+\varepsilon(f({w})-f(\overline{\xi}))}{\overline{w}-\xi+\varepsilon(f(\overline{w})-f(\xi))}f'(\xi)d\xi\\
\nonumber&+\frac{1}{4\pi^2}\int_{\T}\frac{(\overline{w}-\xi)(f({w})-f(\overline{\xi}))-({w}-\overline{\xi})(f(\overline{w})-f(\xi))}{(\overline{w}-\xi)(\overline{w}-\xi+\varepsilon(f(\overline{w})-f(\xi)))}d\xi\\
\nonumber&+\sum_{m=1}^{N-1}\int_{\T}\frac{\overline{\xi}+\varepsilon f(\overline{\xi})}{\varepsilon e^{\frac{i2\pi m}{N}}({\xi}+\varepsilon f({\xi}))-\varepsilon ({\overline{w}}+\varepsilon f({\overline{w}}))+(e^{\frac{i2\pi m}{N}}-1)l}(1+\varepsilon f'(\xi))d\xi.
\end{align*}
Then, finally we find
\begin{align*}
-\overline{J(\varepsilon,f)(w)}=&\frac{1}{4\pi^2}\int_{\T}\frac{{w}-\overline{\xi}+\varepsilon(f({w})-f(\overline{\xi}))}{\overline{w}-\xi+\varepsilon(f(\overline{w})-f(\xi))}f'(\xi)d\xi\\
\nonumber&+\frac{1}{4\pi^2}\int_{\T}\frac{(\overline{w}-\xi)(f({w})-f(\overline{\xi}))-({w}-\overline{\xi})(f(\overline{w})-f(\xi))}{(\overline{w}-\xi)(\overline{w}-\xi+\varepsilon(f(\overline{w})-f(\xi)))}d\xi\\
\nonumber&+\sum_{m=1}^{N-1}\int_{\T}\frac{\overline{\xi}+\varepsilon f(\overline{\xi})}{\varepsilon e^{\frac{i2\pi m}{N}}({\xi}+\varepsilon f({\xi}))-\varepsilon ({\overline{w}}+\varepsilon f({\overline{w}}))+(e^{\frac{i2\pi m}{N}}-1)l}(1+\varepsilon f'(\xi))d\xi\\
=&J(\varepsilon,f)(\overline{w}),
\end{align*}
getting \eqref{J-sym}. That concludes the proof.
\end{proof}

The main idea to find nontrivial roots to $\eqref{F-2}$ is the use of the infinite dimensional Implicit Function theorem around $(\Omega_0,0,0)$ where $\Omega_0$ is the angular velocity of the point vortex configuration in Proposition \ref{Prop-points}, that is, 
\begin{equation}\label{Omega0}
\Omega_0:=\frac{1}{2\pi l^2}\sum_{k=1}^{N-1}\frac{1}{1-e^{\frac{i 2\pi m}{N}}}.
\end{equation}
Hence, first we shall check that $(\Omega_0,0,0)$ is a root of $\eqref{F-2}$, which  is proved in the following proposition.

\begin{pro}\label{Prop-trivial}
The following equation is satisfied $F(\Omega_0,0,0)(w)=0$, for any $w\in\T$, where $\Omega_0$ is defined in \eqref{Omega0}.
\end{pro}
\begin{proof}
From the expression of $F$ in \eqref{F-2} one finds 
\begin{align*}
F(\Omega,0,0)(w)=\textnormal{Re}\left[w\left\{J(0,0)(w)+i\Omega l\right\}\right].
\end{align*}
Moreover, from \eqref{J-exp}, we get
\begin{align*}
J(0,0)(w)=\frac{1}{4\pi^2}\sum_{m=1}^{N-1}\int_{\T}\frac{\overline{\xi}d\xi}{(e^{\frac{i 2\pi m}{N}}-1)l}=-\frac{i}{2\pi l}\sum_{m=1}^{N-1}\frac{1}{1-e^{\frac{i 2\pi m}{N}}}.
\end{align*}
Hence, if $\Omega=\Omega_0$ we trivially obtain that $F(\Omega_0,0,0)(w)=0$ for any $w\in\T$.
\end{proof}

Now we  need to fix the Banach spaces when using the Implicit Function theorem. For $\alpha\in(0,1)$, we define
\begin{align}
X_{1+\alpha}=&\left\{f\in\mathscr{C}^{1+\alpha}(\T), \quad f(w)=\sum_{n\geq 1}a_n w^{-n}, \, a_n\in\R\right\},\label{X}\\
Y_\alpha=&\left\{f\in\mathscr{C}^{\alpha}(\T), \quad f(e^{i\theta})=\sum_{n\geq 2}a_n \sin(n\theta),\, a_n\in\R\right\}.\label{Y}
\end{align}
Let us explain the choice of the previous function spaces. First, note that in Lemma \ref{Lem-F-sym}, we found that if $f\in X_{1+\alpha}$ then
$$
F(\Omega,\varepsilon,f)(e^{i\theta})=\sum_{n\geq 1}f_n \sin(n\theta), \quad \theta\in[0,2\pi].
$$
However, we have the constraint $f_1=0$ in $Y_\alpha$. Such assumption appears when computing the linearized operator of $F$ around $(\Omega_0,0,0)$ because one finds that it is an isomorphism from $X_{1+\alpha}$ to $Y_\alpha$. Moreover, since this is necessary for the Implicit Function theorem, our task now is to fix $\Omega$ to have that $f_1=0$ and thus $F(\Omega,\varepsilon,f)\in Y_\alpha$ is well--defined.

Define $B_{X}(0,\sigma)$ the unit ball of $X$ center in $0$ and radius $\sigma$. The next result reads as follows.

\begin{pro}\label{Prop-omega}
For any $\sigma<1$, $\alpha\in(0,1)$, and some $\varepsilon_0>0$, define the function $\Omega:(-\varepsilon_0,\varepsilon_0)\times B_{X_{1+\alpha}}(0,\sigma)\rightarrow \R$, given by
\begin{equation}\label{Omega-exp}
\Omega(\varepsilon,f):=\frac{i\int_{\T}J(\varepsilon,f)(w)(w-\overline{w})(1+\varepsilon f'(w))dw}{\int_{\T}(1+\varepsilon f'(w))(w-\overline{w})(l+\varepsilon \overline{w}+\varepsilon^2 f(\overline{w}))dw}.
\end{equation}
Hence, it fulfills
\begin{itemize}
\item $\Omega$ is well--defined.
\item $\Omega(0,f)=\Omega_0$, where $\Omega_0$ is defined in \eqref{Omega0}.
\item We get that
\begin{equation}\label{F-prop}
\int_0^{2\pi}F(\Omega(\varepsilon,f),\varepsilon,f)(e^{i\theta})\sin(\theta)=0,
\end{equation}
for any $(\varepsilon,f)\in  (-\varepsilon_0,\varepsilon_0)\times B_{X_{1+\alpha}}(0,\sigma)$ That implies that $f_1$ in \eqref{F-sym} vanishes.
\end{itemize}
\end{pro}
\begin{rem}
We need to include the condition $\varepsilon\in(-\varepsilon_0,\varepsilon_0)$ in order to have that the patches are well--separated. Moreover, we need to assume $\delta<1$ to get that $\Phi$ is bilipschitz and thus everything is well--defined.
\end{rem}
\begin{proof}
We can easily check that $\Omega\in\R$ by using the techniques developed in Lemma \ref{Lem-F-sym}. Let us check the second point. Note that
\begin{align*}
\Omega(0,f)=\frac{i\int_{\T}J(0,f)(w)(w-\overline{w})dw}{l\int_{\T} (w-\overline{w})dw}.
\end{align*}
On the one hand, we obtain
$$
\int_{\T}(w-\overline{w})dw=-2\pi i,
$$
by the Residue theorem, which implies
\begin{align}\label{Omega-2}
\Omega(0,f)=-\frac{\int_{\T}J(0,f)(w)(w-\overline{w})dw}{2\pi l}.
\end{align}
On the other hand, by using \eqref{J-exp} one achieves
\begin{align*}
J(0,f)(w)=&\frac{1}{4\pi^2}\int_{\T} \frac{\overline{w}-\overline{\xi}}{w-\xi}f'(\xi)d\xi\\
&+\frac{1}{4\pi^2}\int_{\T}\frac{(w-\xi)(f(\overline{w})-f(\overline{\xi}))-(\overline{w}-\overline{\xi})(f(w)-f(\xi))}{(w-\xi)^2}d\xi\\
&+\frac{1}{4\pi^2}\sum_{m=1}^{N-1}\frac{1}{l(e^{\frac{i2\pi}{N}}-1)}\int_{\T}\overline{\xi}d\xi.
\end{align*}
Moreover, note that
\begin{align}
\label{int-1}&\int_{\T}\frac{\overline{w}-\overline{\xi}}{w-\xi}f'(\xi)d\xi=0,\\
\label{int-2}&\int_{\T}\frac{(w-\xi)(f(\overline{w})-f(\overline{(\xi}))-(\overline{w}-\overline{\xi})(f(w)-f(\xi))}{(w-\xi)^2}d\xi=2i\int_{\T}\frac{\textnormal{Im}\left[\overline{(f(w)-f(\xi))(w-\xi)}\right]}{(w-\xi)^2}d\xi=0,
\end{align}
using again the Residue theorem. That implies
\begin{align*}
J(0,f)(w)=&\frac{1}{4\pi^2}\sum_{m=1}^{N-1}\frac{2\pi i}{l(e^{\frac{i2\pi}{N}}-1)}=-\frac{i}{2\pi l}\sum_{m=1}^{N-1}\frac{1}{1-e^{\frac{i2\pi}{N}}}.
\end{align*}
Inserting this into \eqref{Omega-2} we get $\Omega(0,f)=\Omega_0$. Finally, note that \eqref{Omega-exp} comes from imposing \eqref{F-prop} to \eqref{F-2}.
\end{proof}

Once we have defined $\Omega$ in terms of $(\varepsilon,f)$, we can define $\tilde{F}(\varepsilon,f)$ as
\begin{equation}\label{F-tilde}
\tilde{F}(\varepsilon,f)(w)=F(\Omega(\varepsilon,f),\varepsilon,f)(w),
\end{equation}
and work with $\tilde{F}$ instead of $F$. Note that from Proposition \ref{Prop-omega} we get that
$$
\tilde{F}(\varepsilon,f)(e^{i\theta})=\sum_{n\geq 2}f_n\sin(n\theta),
$$
with $\theta\in[0,2\pi]$ and $f_n\in\R$. Then, in order to check that $\tilde{F}:(-\varepsilon_0,\varepsilon_0)\times X_{1+\alpha}\rightarrow Y_{\alpha}$ is well--defined it remains to prove the regularity of $Y_{\alpha}$. For this reason we must work with singular integrals of the type
\begin{equation}\label{operator-T}
\mathcal{T}(f)(w)=\int_{\T} K(w,\xi)f(\xi)\, \, d\xi,\quad w\in\T,
\end{equation}
where $K:\T\times \T\rightarrow\C$ being smooth off the diagonal. The next result focuses on the smoothness of the last operator, whose proof can be found in \cite{Hassa-Hmi}. See also \cite{Helms, Kress, LiebLoss}.

\begin{lem}\label{Lem-pottheory}
Let $0\leq \alpha<1$ and consider $K:\T\times\T\rightarrow \C$ with the following properties. There exists $C_0>0$ such that
\begin{itemize}
\item[(i)] $K$ is measurable on $\T\times\T\setminus\{(w,w), w\in\T\}$ and 
$$
|K(w,\xi)|\leq \frac{C_0}{|w-\xi|^\alpha}, \quad \forall w\neq \xi\in\T.
$$
\item[(ii)] For each $\xi\in\T$, $w\mapsto K(w,\xi)$ is differentiable in $\T\setminus\{\xi\}$ and 
$$
|\partial_w K(w,\xi)|\leq \frac{C_0}{|w-\xi|^{1+\alpha}}, \quad \forall w\neq \xi\in\T.
$$
\end{itemize}
Then,
\begin{enumerate}
\item The operator $\mathcal{T}$ defined by \eqref{operator-T} is continuous from $L^{\infty}(\T)$ to $C^{1-\alpha}(\T)$. More precisely, there exists a constant $C_{\alpha}$ depending only on $\alpha$ such that
$$
\|\mathcal{T}(f)\|_{1-\alpha}\leq C_{\alpha}C_0\|f\|_{L^{\infty}}.
$$
\item For $\alpha=0$, the operator $\mathcal{T}$ is continuous from $L^{\infty}(\T)$ to $C^{\beta}(\T)$, for any $0<\beta<1$. That is, there exists a constant $C_{\beta}$ depending only on $\beta$ such that
$$
\|\mathcal{T}(f)\|_{\beta}\leq C_{\beta}C_0\|f\|_{L^{\infty}}.
$$
\end{enumerate}
\end{lem}

By virtue of the previous lemma, we obtain that $J(\varepsilon,f):(-\varepsilon_0,\varepsilon_0)\times B_{X_{1+\alpha}}(0,\sigma)\rightarrow \mathscr{C}^{\alpha}(\T)$ is well--defined and $\mathscr{C}^1$. That implies the well--definition of $\tilde{F}$. We omit here the proof due to its similarity with the works \cite{H-M, G-KVS}.

\begin{pro}\label{Prop-welldef}
For any $\sigma<1$, $\alpha\in(0,1)$, and some $\varepsilon_0>0$, $\tilde{F}:(-\varepsilon_0,\varepsilon_0)\times B_{X_{1+\alpha}}(0,\sigma)\rightarrow Y_{\alpha}$, is well--defined and $\mathscr{C}^1$.
\end{pro}
%
%

Finally, in the next result we use the infinite dimensional Implicit Function theorem to $\tilde{F}$ in order to get nontrivial roots and thus to conclude the main result of this work.
\begin{theo}\label{Th-euler}
Consider $l\in\R^\star$ and $N\geq 2$. Then, there exists $\varepsilon_0>0$ with the following property. For all $\varepsilon\in(0,\varepsilon_0)$, there is a simply--connected bounded domain $D^\varepsilon$, with center of masses $l$, such that
$$
\omega_0(x)=\frac{1}{\pi \varepsilon^2}\sum_{m=0}^{N-1}{\bf 1}_{e^{\frac{i 2\pi m}{N}}D^\varepsilon} (x),
$$
defines a rotating solution of \eqref{Euler}, with some constant angular velocity $\Omega(\varepsilon)$. Moreover, $D^\varepsilon$ is at least $\mathscr{C}^1$.
\end{theo}
\begin{proof}
From Proposition \ref{Prop-welldef}, we have that $\tilde{F}:(-\varepsilon_0,\varepsilon_0)\times B_{X_{1+\alpha}}(0,\sigma)\rightarrow Y_{\alpha}$ is well--defined and $\mathscr{C}^1$. Moreover, from Propositions \ref{Prop-trivial} and \ref{Prop-omega} one has that $\tilde{F}(0,0)(w)=0$, for any $w\in\T$.

Our aim is to apply the infinite dimensional Implicit Function theorem. Hence, we compute the linearized operator of $\tilde{F}$ around $(0,0)$. Note first that $\partial_f \Omega(0,0)=0$ by using Proposition \ref{Prop-omega}, and then
$$
\partial_f \tilde{F}(\varepsilon,f)h(w)=\partial _f F(\Omega_0,0,0)h(w).
$$
First, let us compute $\partial_f J(0,0)$ by using \eqref{J-exp}:
\begin{align*}
\partial_f J(0,0)h(w)=\frac{1}{4\pi^2}\int_{\T}\frac{\overline{w}-\overline{\xi}}{w-\xi}h'(\xi)d\xi+\frac{1}{4\pi^2}\int_{\T}\frac{(w-\xi)(h(\overline{w})-h(\overline{(\xi}))-(\overline{w}-\overline{\xi})(h(w)-h(\xi))}{(w-\xi)^2}d\xi.
\end{align*}
Moreover, \eqref{int-1}--\eqref{int-2} amounts to
$$
\partial_f J(0,0)h=0.
$$
Then, using \eqref{F-2} we find
\begin{align*}
\partial_f \tilde{F}(0,0)h(w)=&-\frac{1}{2\pi}\textnormal{Re}\left[ih'(w)\right]\\
&+\lim_{\varepsilon\rightarrow 0}\textnormal{Re}\left[\varepsilon wh'(w)\left\{J(0,0)(w)+i\Omega_0\varepsilon \overline{w}+i\Omega_0 l\right\}\right]\\
&+\textnormal{Re}\left[w\left\{\partial_f J(0,0)h(w)\right\}\right]\\
=&-\frac{1}{2\pi}\textnormal{Re}\left[ih'(w)\right],
\end{align*}
which is an isomorphism from $X_{1+\alpha}$ into $Y_\alpha$. Using then the Implicit Function theorem, we conclude the proof.
\end{proof}

\section{Vortex patch polygon for generalized surface quasi--geostrophic equations}\label{Sec-gsqg}
This section aims to generalize the previous result on Theorem \ref{Th-euler} for the generalized surface quasi--geostrophic equations with $\beta\in(0,1)$. First, we shall arrive to the equation that generalize the rotating polygon.

Similarly to Section \ref{Sec-Euler}, we would like to find domains $D_m^\varepsilon$, for $m=0,\dots N-1$, $N\geq 2$ and $\varepsilon>0$, such that the initial data
\begin{equation}\label{initial-data-g}
\theta_{0,\varepsilon}(x)=\frac{1}{\pi \varepsilon^2}\sum_{m=0}^{N-1}{\bf 1}_{D_m^\varepsilon}(x),
\end{equation}
evolves as a rotation for \eqref{gSQG}. Then, consider that there exists $\Omega\in\R$ and $\varepsilon>0$ such that
$$
\theta(t,x)=\theta_{0,\varepsilon}(e^{-i\Omega t}x),
$$
and
\begin{equation}\label{Dm-g}
D_m^\varepsilon=e^{i\frac{2\pi m}{N}}D_0^\varepsilon,
\end{equation}
for some bounded simply--connected domain $D_0^\varepsilon$. Moreover, assume that the center of masses of $D_0^\varepsilon$ is $l\in\R^\star$ and parametrize $D_0^\varepsilon-l$ by a conformal map from $\T$ into $\partial D_0^\varepsilon-l$ of the type
\begin{equation}\label{conf-map-g}
\Phi(w)=\varepsilon\left(w+\varepsilon^{1+\beta} f(w)\right), \quad f(w)=\sum_{n\geq 1}a_n w^{-n}, \quad a_n\in\R, w\in\T.
\end{equation}
Note that the scaling for $\Phi$ in $\varepsilon$ is different from \eqref{conformal-map}, which is the one for the Euler equations. Indeed it depends on the singularity of the kernel of the velocity field.

Hence, the equation that characterizes this kind of solutions is given by
\begin{equation}\label{eq1-g}
\textnormal{Re}\left[\left\{\overline{v_0(\Phi(w)+l)}+i\Omega(\overline{\Phi(w)}+l)\right\}w\Phi'(w)\right]=0, \quad w\in\T.
\end{equation}
We omit here the details of the deduction of the previous equation by similarity to \eqref{eq-1}.

In the following we give a simplified expression for the velocity field in order to arrive at the desired formulation. There, we use Taylor formula in a suitable manner to work better with the singularity coming from $\varepsilon$. Hence, the velocity field reads as follows.

\begin{pro}\label{Prop-vel-g}
The velocity field $v_0$ associated to the initial data \eqref{initial-data-g} agrees with
\begin{align*}
v_0(\Phi(w)+l)=&\frac{C_\beta}{\pi \varepsilon^{1+\beta}}\int_{\T}\frac{d\xi}{|w-\xi|^\beta}+J(\varepsilon,f)(w),
\end{align*}
where
\begin{align}
&J(\varepsilon,f)(w)=-\frac{\beta C_\beta}{\pi}\int_{\T}\int_0^1\frac{\textnormal{Re}\left[(w-\xi)\overline{(f(w)-f(\xi))}\right]+t\varepsilon^{1+\beta}|f(w)-f(\xi)|^2}{|w-\xi+t\varepsilon^{1+\beta}(f(w)-f(\xi))|^{2+\beta}}dt d\xi\nonumber\\
&+\frac{C_\beta}{\pi}\int_{\T}\frac{f'(\xi)}{\left|w-\xi+\varepsilon^{1+\beta} (f(w)-f(\xi))\right|^\beta}d\xi-\frac{\beta C_\beta}{\pi}\sum_{m=1}^{N-1}e^{\frac{i2\pi m}{N}}\nonumber\\
&\times\int_{\T}\Bigg\{\int_0^1\frac{\textnormal{Re}\left[l(1-e^{\frac{i2\pi m}{N}})\overline{(w-e^{\frac{i2\pi m}{N}}\xi+\varepsilon^{1+\beta}(f(w)-e^{\frac{i2\pi m}{N}} f(\xi)))}\right]}{|l(1-e^{\frac{i2\pi m}{N}})+t(\Phi(w)-e^{\frac{i2\pi m}{N}}\Phi(\xi))|^{2+\beta}}\nonumber\\
&+\frac{t\varepsilon|w-e^{\frac{i2\pi m}{N}}\xi+\varepsilon^{1+\beta}(f(w)-e^{\frac{i2\pi m}{N}} f(\xi))|^2}{|l(1-e^{\frac{i2\pi m}{N}})+t(\Phi(w)-e^{\frac{i2\pi m}{N}}\Phi(\xi))|^{2+\beta}}\Bigg\}dt d\xi\nonumber\\
&+\frac{C_\beta \varepsilon^\beta}{\pi}\sum_{m=1}^{N-1}e^{\frac{i2\pi m}{N}}\int_{\T}\frac{f'(\xi)}{\left|\Phi(w)-e^{\frac{i2\pi m}{N}}\Phi(\xi)+l(1-e^{\frac{i2\pi m}{N}})\right|^\beta}d\xi\nonumber\\
=:& (J_1+J_2+J_3+J_4)(\varepsilon,f)(w), \label{J-g}
\end{align}
for any $w\in\T$.
\end{pro}
\begin{proof}
The velocity field associated to \eqref{initial-data-g} reads as
\begin{align*}
v_0(x)=\frac{C_\beta i}{\pi\varepsilon^2}\sum_{m=0}^{N-1}\int_{D_m^\varepsilon}\frac{x-y}{|x-y|^{2+\beta}}dA(y),
\end{align*}
for $x\in\partial D_0^\varepsilon.$ Using Stokes theorem, we can write it as follows
\begin{align*}
v_0(x)=\frac{C_\beta}{\pi\varepsilon^2}\sum_{m=0}^{N-1}\int_{\partial D_m^\varepsilon}\frac{dy}{|x-y|^\beta},
\end{align*}
for $x\in\partial D_0^\varepsilon.$ Then, taking $x=\Phi(w)+l$ and using \eqref{Dm-g}, it agrees with
\begin{align*}
v_0(\Phi(w)+l)=\frac{C_\beta}{\pi\varepsilon^2}\int_{\partial D_0^\varepsilon}\frac{dy}{|\Phi(w)+l-y|^\beta}+\frac{C_\beta}{\pi\varepsilon^2}\sum_{m=1}^{N-1}e^{\frac{i2\pi m}{N}}\int_{\partial D_0^\varepsilon}\frac{dy}{|\Phi(w)+l-e^{\frac{i2\pi m}{N}}y|},
\end{align*}
for $w\in\T$. Hence
\begin{align*}
v_0(\Phi(w)+l)=&\frac{C_\beta}{\pi\varepsilon^2}\int_{\T}\frac{\Phi'(\xi)}{|\Phi(w)-\Phi(\xi)|^\beta}d\xi\\
&+\frac{C_\beta}{\pi\varepsilon^2}\sum_{m=1}^{N-1}e^{\frac{i2\pi m}{N}}\int_{\T}\frac{\Phi'(\xi)}{|\Phi(w)-e^{\frac{i2\pi m}{N}}\Phi(\xi)+l(1-e^{\frac{i2\pi m}{N}})|^\beta}d\xi,
\end{align*}
for $w\in\T$. Now using the definition of $\Phi$, one gets
\begin{align*}
v_0(\Phi(w)+l)=&\frac{C_\beta}{\pi\varepsilon^{1+\beta}}\int_{\T}\frac{d\xi}{\left|(w-\xi)+\varepsilon^{1+\beta} (f(w)-f(\xi))\right|^\beta}\\
&+\frac{C_\beta}{\pi}\int_{\T}\frac{f'(\xi)}{\left|w-\xi+\varepsilon^{1+\beta} (f(w)-f(\xi))\right|^\beta}d\xi\\
&+\frac{C_\beta}{\pi\varepsilon}\sum_{m=1}^{N-1}e^{\frac{i2\pi m}{N}}\int_{\T}\frac{d\xi}{\left|\Phi(w)-e^{\frac{i2\pi m}{N}}\Phi(\xi)+l(1-e^{\frac{i2\pi m}{N}})\right|^\beta}\\
&+\frac{C_\beta \varepsilon^\beta}{\pi}\sum_{m=1}^{N-1}e^{\frac{i2\pi m}{N}}\int_{\T}\frac{f'(\xi)}{\left|\Phi(w)-e^{\frac{i2\pi m}{N}}\Phi(\xi)+l(1-e^{\frac{i2\pi m}{N}})\right|^\beta}d\xi\\
=:& I_1(\varepsilon,f)(w)+I_2(\varepsilon,f)(w)+I_3(\varepsilon,f)(w)+I_4(\varepsilon,f)(w),
\end{align*}
for $w\in\T$. 

In order to overcome the singularity in $\varepsilon$ of $I_1$ and $I_3$, we use Taylor formula. Indeed, taking any complex numbers $A$ and $B$ such that $|B|<|A|$ one has
\begin{equation}\label{taylor}
\frac{1}{|A+B|^\beta}=\frac{1}{|A|^\beta}-\beta\int_0^1\frac{\textnormal{Re}\left[A\overline{B}\right]+t|B|^2}{|A+tB|^{2+\beta}}dt.
\end{equation}
In the case of $I_1$, take $A=w-\xi$ and $B=\varepsilon^{1+\beta}(f(w)-f(\xi))$, implying
\begin{align*}
I_1(\varepsilon,f)(w)=&\frac{C_\beta}{\pi \varepsilon^{1+\beta}}\int_{\T}\frac{d\xi}{|w-\xi|^\beta}-\frac{\beta C_\beta}{\pi}\int_{\T}\int_0^1\frac{\textnormal{Re}\left[(w-\xi)\overline{(f(w)-f(\xi))}\right]+t\varepsilon^{1+\beta}|f(w)-f(\xi)|^2}{|w-\xi+t\varepsilon^{1+\beta}(f(w)-f(\xi))|^{2+\beta}}dt d\xi.
\end{align*}
In the same way, we can work with $I_3$ considering $z_1=l(1-e^{\frac{i2\pi m}{N}})$ and $z_2=\Phi(w)-e^{\frac{i2\pi m}{N}}\Phi(\xi)$ and thus obtaining
\begin{align*}
&I_3(\varepsilon,f)(w)=\frac{C_\beta}{\pi \varepsilon}\sum_{m=1}^{N-1}e^{\frac{i2\pi m}{N}}\int_{\T}\frac{d\xi}{|l(1-e^{\frac{i2\pi m}{N}})|^\beta}-\frac{\beta C_\beta}{\pi}\sum_{m=1}^{N-1}e^{\frac{i2\pi m}{N}}\\
&\times\int_{\T}\int_0^1\Bigg\{\frac{\textnormal{Re}\left[l(1-e^{\frac{i2\pi m}{N}})\overline{(w-e^{\frac{i2\pi m}{N}}\xi+\varepsilon^{1+\beta}(f(w)-e^{\frac{i2\pi m}{N}}f(\xi)))}\right]}{|l(1-e^{\frac{i2\pi m}{N}})+t(\Phi(w)-e^{\frac{i2\pi m}{N}}\Phi(\xi))|^{2+\beta}}\\
&+\frac{t\varepsilon|w-e^{\frac{i2\pi m}{N}}\xi+\varepsilon^{1+\beta}(f(w)-e^{\frac{i2\pi m}{N}}f(\xi))|^2}{|l(1-e^{\frac{i2\pi m}{N}})+t(\Phi(w)-e^{\frac{i2\pi m}{N}}\Phi(\xi))|^{2+\beta}}\Bigg\}dt d\xi\\
=&-\frac{\beta C_\beta}{\pi}\sum_{m=1}^{N-1}e^{\frac{i2\pi m}{N}}
 \int_{\T}\int_0^1\Bigg\{\frac{\textnormal{Re}\left[l(1-e^{\frac{i2\pi m}{N}})\overline{(w-e^{\frac{i2\pi m}{N}}\xi+\varepsilon^{1+\beta}(f(w)-e^{\frac{i2\pi m}{N}}f(\xi)))}\right]}{|l(1-e^{\frac{i2\pi m}{N}})+t(\Phi(w)-e^{\frac{i2\pi m}{N}}\Phi(\xi))|^{2+\beta}}\\
&+\frac{t\varepsilon|w-e^{\frac{i2\pi m}{N}}\xi+\varepsilon^{1+\beta}(f(w)-e^{\frac{i2\pi m}{N}}f(\xi))|^2}{|l(1-e^{\frac{i2\pi m}{N}})+t(\Phi(w)-e^{\frac{i2\pi m}{N}}\Phi(\xi))|^{2+\beta}}\Bigg\}dt d\xi.
\end{align*}
Hence, we can write the velocity field as
\begin{align*}
v_0&(\Phi(w)+l)=\frac{C_\beta}{\pi \varepsilon^{1+\beta}}\int_{\T}\frac{d\xi}{|w-\xi|^\beta}\\
&-\frac{\beta C_\beta}{\pi}\int_{\T}\int_0^1\frac{\textnormal{Re}\left[(w-\xi)\overline{(f(w)-f(\xi))}\right]+t\varepsilon^{1+\beta}|f(w)-f(\xi)|^2}{|w-\xi+t\varepsilon^{1+\beta}(f(w)-f(\xi))|^{2+\beta}}dt d\xi\\
&+\frac{C_\beta}{\pi}\int_{\T}\frac{f'(\xi)}{\left|w-\xi+\varepsilon^{1+\beta} (f(w)-f(\xi))\right|^\beta}d\xi\\
&-\frac{\beta C_\beta}{\pi}\sum_{m=1}^{N-1}e^{\frac{i2\pi m}{N}}\int_{\T}\int_0^1\Bigg\{\frac{\textnormal{Re}\left[l(1-e^{\frac{i2\pi m}{N}})\overline{(w-e^{\frac{i2\pi m}{N}}\xi+\varepsilon^{1+\beta}(f(w)-e^{\frac{i2\pi m}{N}}f(\xi)))}\right]}{|l(1-e^{\frac{i2\pi m}{N}})+t(\Phi(w)-e^{\frac{i2\pi m}{N}}\Phi(\xi))|^{2+\beta}}\\
&+\frac{t\varepsilon|w-e^{\frac{i2\pi m}{N}}\xi+\varepsilon^{1+\beta}(f(w)-e^{\frac{i2\pi m}{N}}f(\xi))|^2}{|l(1-e^{\frac{i2\pi m}{N}})+t(\Phi(w)-e^{\frac{i2\pi m}{N}}\Phi(\xi))|^{2+\beta}}\Bigg\}dt d\xi\\
&+\frac{C_\beta \varepsilon^\beta}{\pi}\sum_{m=1}^{N-1}e^{\frac{i2\pi m}{N}}\int_{\T}\frac{f'(\xi)}{\left|\Phi(w)-e^{\frac{i2\pi m}{N}}\Phi(\xi)+l(1-e^{\frac{i2\pi m}{N}})\right|^\beta}d\xi\\
=:&\frac{C_\beta}{\pi \varepsilon^{1+\beta}}\int_{\T}\frac{d\xi}{|w-\xi|^\beta}+J(\varepsilon,f)(w).
\end{align*}
That concludes the proof.
\end{proof}

By virtue of Proposition \ref{Prop-vel-g}, one can simplify the expression for the velocity field \eqref{eq1-g}. Indeed, from \cite{Hassa-Hmi} one has
\begin{equation*}
\int_{\T}\frac{d\xi}{|w-\xi|^{\beta}}=\frac{2\pi i\beta \Gamma(1-\beta)}{(2-\beta)\Gamma^2\left(1-\frac{\beta}{2}\right)}w,
\end{equation*}
and then
$$
\textnormal{Re}\left[w\overline{\int_{\T}\frac{d\xi}{|w-\xi|^\beta}}\right]=0.
$$

Hence, using the decomposition of $v_0$ given in Proposition \ref{Prop-vel-g}, one can write \eqref{eq1-g} as 
\begin{align}\label{equation-general}
\nonumber F_{\beta}(\Omega,\varepsilon,f)(w):=&\textnormal{Re}\left[\left\{\overline{J(\varepsilon,f)(w)}+i\Omega(\overline{\Phi(w)}+l)\right\}w(1+\varepsilon^{1+\beta} f'(w))\right]\\
&+\mu_\beta\textnormal{Im}\left[f'(w)\right]=0, \quad w\in\T,
\end{align}
where $J$ is defined in \eqref{J-g} and
$$
\mu_\beta:=\frac{C_\beta 2 \beta \Gamma(1-\beta)}{(2-\beta)\Gamma^2\left(1-\frac{\beta}{2}\right)}.
$$

The analogue of Lemma \ref{Lem-F-sym} is found here obtaining that $F_{\beta}$ can be written as a sum of sines. We omit the proof for its similarity.
\begin{lem}\label{Lem-F-sym-g}
If $(\Omega,\varepsilon)\in\R^2$ and $f\in \mathscr{C}^1$ satisfies 
\begin{equation}\label{f-exp-g}
f(w)=\sum_{n\geq 1}a_nw^{-n},\quad a_n\in\R,w\in\T,
\end{equation}
hence 
\begin{equation}\label{F-sym2-g}
F_{\beta}(\Omega,\varepsilon,f)(\overline{w})=-F_{\beta}(\Omega,\varepsilon,f)(w), \quad w\in\T.
\end{equation} 
That implies that
\begin{equation}\label{F-sym-g}
F_{\beta}(\Omega,\varepsilon,f)(e^{i\theta})=\sum_{n\geq 1}f_n \sin(\theta), \quad f_n\in\R, \theta\in[0,2\pi].
\end{equation}
\end{lem}

In the following result, we shall check that $(\Omega_0,0,0)$ is a trivial root of $F_{\beta}$ as happens for the formulation in the Euler equations. Recall that $\Omega_0$ in this case is given in Remark \ref{rem-sqg} by
\begin{equation}\label{Omega0-g}
\Omega_0=\frac{\beta C_\beta}{l^{2+\beta}}\sum_{k=1}^{N-1}\frac{1-e^{\frac{i2\pi k}{N}}}{|1-e^{\frac{i2\pi k}{N}}|^{2+\beta}}.
\end{equation}
\begin{pro}\label{Prop-trivial-g}
The following equation is verified $F_{\beta}(\Omega_0,0,0)(w)=0$, for any $w\in\T$, where $\Omega_0$ is defined in \eqref{Omega0-g}.
\end{pro}
\begin{proof}
First, note that
$$
F_{\beta}(\Omega_0,0,0)(w)=\textnormal{Re}\left[\left\{\overline{J(0,0)(w)}+i\Omega_0 l\right\}w\right].
$$
We use the decomposition of $J$ given in \eqref{J-g} in order to compute $J(0,0)$. Indeed, one finds
$$
J(0,0)=J_3(0,0).
$$
Moreover,
\begin{align*}
J_3(0,0)=&-\frac{l\beta C_\beta}{\pi}\sum_{m=1}^{N-1}\frac{e^{\frac{i2\pi m}{N}}}{l^{2+\beta}|1-e^{\frac{i2\pi m}{N}}|^{2+\beta}}\int_{\T}\textnormal{Re}\left[(1-e^{\frac{i2\pi m}{N}})\overline{(w-e^{\frac{i2\pi m}{N}}\xi)}\right]d\xi\\
=&-\frac{l\beta C_\beta}{2\pi}\sum_{m=1}^{N-1}\frac{e^{\frac{i2\pi m}{N}}(1-e^{\frac{i2\pi m}{N}})}{l^{2+\beta}|1-e^{\frac{i2\pi m}{N}}|^{2+\beta}}\int_{\T}\overline{(w-e^{\frac{i2\pi m}{N}}\xi)}d\xi\\
=&il\beta C_\beta\sum_{m=1}^{N-1}\frac{(1-e^{\frac{i2\pi m}{N}})}{l^{2+\beta}|1-e^{\frac{i2\pi m}{N}}|^{2+\beta}}.
\end{align*}
Then, it is clear that
$$
\overline{J_3(0,0)(w)}+i\Omega_0l=0,
$$
implying
$$
F_{\beta}(\Omega_0,0,0)=0.
$$
\end{proof}

Here, we use the Banach spaces defined in \eqref{X}--\eqref{Y} with a suitable exponent for the H\"older spaces given by the singularity of the kernel. As for Euler equations, we need to fix here $\Omega$ depending on  $(\varepsilon,f)$ to have that $F_{\beta}\in Y_{1-\beta}$ if $(\varepsilon,f)\in\R\times X_{2-\beta}$. That is the analogue result to Proposition \ref{Prop-omega}.

\begin{pro}\label{Prop-omega-g}
For any $\sigma<1$,  and some $\varepsilon_0>0$, define the function $\Omega:(-\varepsilon_0,\varepsilon_0)\times B_{X_{2-\beta}}(0,\sigma)\rightarrow \R$, given by
\begin{equation}\label{Omega-exp-g}
\Omega(\varepsilon,f):=\frac{i\int_{\T}\overline{J(\varepsilon,f)(w)}(w-\overline{w})(1+\varepsilon^{1+\beta}f'(w))dw}{\int_{\T}(1+\varepsilon^{1+\beta} f'(w))(w-\overline{w})(l+\varepsilon \overline{w}+\varepsilon^{2+\beta} f(\overline{w}))dw}.
\end{equation}
Hence, it fulfills
\begin{itemize}
\item $\Omega$ is well--defined.
\item $\Omega(0,f)=\Omega_0$, where $\Omega_0$ is defined in \eqref{Omega0-g}.
\item We get that
\begin{equation}\label{F-prop-g}
\int_0^{2\pi}F_{\beta}(\Omega(\varepsilon,f),\varepsilon,f)(e^{i\theta})\sin(\theta)=0,
\end{equation}
for any $(\varepsilon,f)\in  (-\varepsilon_0,\varepsilon_0)\times B_{X_{2-\beta}}(0,\sigma)$ That implies that $f_1$ in \eqref{F-sym} vanishes.
\end{itemize}
\end{pro}

\begin{proof}
It is easy to check that $\Omega\in\R$ and thus it is well--defined. Let us verify the second statement. Note that
$$
\Omega(0,f)=\frac{i\int_{\T}\overline{J(0,f)(w)}(w-\overline{w})dw}{l\int_{\T}(w-\overline{w})dw}=-\frac{\int_{\T}\overline{J(0,f)(w)}(w-\overline{w})dw}{2\pi l}.
$$
From the decomposition of $J$ in \eqref{J-g} one has
\begin{align*}
J(0,f)(w)=&J_1(0,f)(w)+J_2(0,f)(w)+J_3(0,f)(w)\\
=&-\frac{\beta C_\beta}{\pi}\int_{\T}\frac{\textnormal{Re}\left[(w-\xi)\overline{(f(w)-f(\xi))}\right]}{|w-\xi|^{2+\beta}} d\xi\nonumber+\frac{C_\beta}{\pi}\int_{\T}\frac{f'(\xi)}{\left|w-\xi\right|^\beta}d\xi\\
&-\frac{\beta C_\beta}{\pi}\sum_{m=1}^{N-1}e^{\frac{i2\pi m}{N}}\int_{\T}\frac{\textnormal{Re}\left[l(1-e^{\frac{i2\pi m}{N}})\overline{(w-e^{\frac{i2\pi m}{N}}\xi)}\right]}{|l(1-e^{\frac{i2\pi m}{N}})|^{2+\beta}}d\xi.
\end{align*}

From \cite{H-M}, we can compute the previous integrals as
\begin{align}
\frac{\beta C_\beta}{2\pi i}\int_{\T}\frac{\textnormal{Re}\left[(w-\xi)\overline{(f(w)-f(\xi))}\right]}{|w-\xi|^{2+\beta}} d\xi=&\sum_{n\geq 1}a_n\left(\gamma_{1,n}w^{n+2}+\gamma_{2,n}\overline{w}^n\right),\label{int-1-g}\\
\frac{C_\beta}{2\pi i}\int_{\T}\frac{f'(\xi)}{\left|w-\xi\right|^\beta}d\xi=&-\frac{C_\beta \Gamma(1-\beta)}{\Gamma^2(1-\beta/2)}\sum_{n\geq 1}na_n\frac{\left(\frac{\beta}{2}\right)_n}{\left(1-\frac{\beta}{2}\right)_n}\overline{w}^n,\label{int-2-g}
\end{align}
where
\begin{align*}
\gamma_{1,n}:=&\frac{\beta (1+\beta/2) C_\beta \Gamma(1-\beta)}{2(2-\beta)\Gamma^2(1-\beta/2)}\left(1-\frac{\left(2+\frac{\beta}{2}\right)_n}{\left(2-\frac{\beta}{2}\right)_n}\right),\\
\gamma_{2,n}:=&-\frac{\beta C_\beta \Gamma(1-\beta)}{4\Gamma^2(1-\beta/2)}\left(1-\frac{\left(\frac{\beta}{2}\right)_n}{\left(-\frac{\beta}{2}\right)_n}\right).
\end{align*}
Hence, by the Residue theorem one finds
$$
\int_{\T}\overline{(J_1(0,f)+J_2(0,f))}(w-\overline{w})dw=0.
$$
Moreover, $J_3(0,f)$ reads as
\begin{align*}
J_3(0,f)=\beta C_\beta i\sum_{m=1}^{N-1}\frac{l(1-e^{\frac{i2\pi m}{N}})}{|l(1-e^{\frac{i2\pi m}{N}})|^{2+\beta}}=il\Omega_0,
\end{align*}
implying
\begin{align*}
\Omega(0,f)=-\frac{\int_{\T}\overline{J_3(0,f)(w)}(w-\overline{w})dw}{2\pi l}=il\Omega_0\frac{\int_{\T}(w-\overline{w})dw}{2\pi l}=\Omega_0.
\end{align*}
Finally, condition \eqref{Omega-exp-g} amouns to \eqref{F-prop-g}.
\end{proof}

Then, we can work with $\tilde{F_{\beta}}$ defined by 
$$
\tilde{F_{\beta}}(\varepsilon,f)(w)=F(\Omega(\varepsilon,f),\varepsilon,f)(w).
$$
instead of $F_{\beta}$. In this way, we can ensure the following
$$
\tilde{F_{\beta}}(\varepsilon,f)(e^{i\theta})=\sum_{n\geq 2}f_n\sin(n\theta),
$$
for $\theta\in[0,2\pi]$ and $f_n\in\R$. Thus, it remains to check the regularity properties of $\tilde{F_{\beta}}$.

Using Lemma \ref{Lem-pottheory} we can check that $J(\varepsilon,f):(-\varepsilon_0,\varepsilon_0)\times B_{X_{2-\beta}}\rightarrow \mathscr{C}^{1-\beta}(\T)$ is well--defined and $\mathscr{C}^1$ implying the well--definition of $\tilde{F_{\beta}}$. 

\begin{pro}\label{Prop-welldef-g}
For any $\sigma<1$, and some $\varepsilon_0>0$, $\tilde{F_{\beta}}:(-\varepsilon_0,\varepsilon_0)\times B_{X_{2-\beta}}(0,\sigma)\rightarrow Y_{1-\beta}$ is well--defined and $\mathscr{C}^1$.
\end{pro}

Finally, we can use the infinite dimensional Implicit Function theorem to $\tilde{F_{\beta}}$ to get the existence of the rotating polygon for patches. That is presented in the following theorem.

\begin{theo}\label{Th-sqg}
Consider $l\in\R^\star$ and $N\geq 2$. Then, there exists $\varepsilon_0>0$ with the following property. For all $\varepsilon\in(0,\varepsilon_0)$, there is a simply--connected bounded domain $D^\varepsilon$, with center of masses $l$, such that
$$
\theta_0(x)=\frac{1}{\pi \varepsilon^2}\sum_{m=0}^{N-1}{\bf 1}_{e^{\frac{i 2\pi m}{N}}D^\varepsilon} (x),
$$
defines a rotating solution of the generalized surface quasi--geostrophic equations for $\beta\in(0,1)$, with some constant angular velocity $\Omega(\varepsilon)$. Moreover, $D^\varepsilon$ is at least $\mathscr{C}^1$.
\end{theo}
\begin{proof}
From Proposition \ref{Prop-welldef-g} one has that $\tilde{F_{\beta}}:(-\varepsilon_0,\varepsilon_0)\times B_{X_{2-\beta}}(0,\sigma)\rightarrow Y_{1-\beta}$ is well--defined and $\mathscr{C}^1$. Moreover, from Propositions \ref{Prop-trivial-g} and \ref{Prop-omega-g} one gets $\tilde{F_{\beta}}(0,0)(w)=0$, for any $w\in\T$.

The next step to use the infinite dimensional Implicit Function theorem is showing that the linearized operator of $\tilde{F_{\beta}}$ around $(0,0)$ is an isomorphism. Furthermore, since $\partial_f \Omega(0,0)=0$ by Proposition \ref{Prop-omega-g} one achieves
$$
\partial_f \tilde{F_{\beta}}(0,0)h(w)=\partial_f {F_{\beta}}(\Omega_0,0,0)h(w).
$$
First, let us start computing $\partial_f J(0,0)$:
\begin{align*}
\partial_f J(0,0)h(w)=&-\frac{\beta C_\beta}{\pi}\int_{\T}\frac{\textnormal{Re}\left[(w-\xi)\overline{(h(w)-h(\xi))}\right]}{|w-\xi|^{2+\beta}} d\xi\nonumber+\frac{C_\beta}{\pi}\int_{\T}\frac{h'(\xi)}{\left|w-\xi\right|^\beta}d\xi,
\end{align*}
and then
\begin{align*}
\partial_f \tilde{F_{\beta}}(0,0)h(w)=\mu_{\beta}\textnormal{Im}[h'(w)]+\textnormal{Re}\left[\overline{K(h)(w)}w\right]+\partial_f \Omega(0,0) l\textnormal{Re}[iw],
\end{align*}
where
$$
K(h)(w):=-\frac{\beta C_\beta}{\pi}\int_{\T}\frac{\textnormal{Re}\left[(w-\xi)\overline{(h(w)-h(\xi))}\right]}{|w-\xi|^{2+\beta}} d\xi\nonumber+\frac{C_\beta}{\pi}\int_{\T}\frac{h'(\xi)}{\left|w-\xi\right|^\beta}d\xi.
$$
However, by virtue of Proposition \ref{Prop-omega-g} one has $\partial_f \Omega(0,0)=0$ and hence the last term vanishes having
\begin{align*}
\partial_f \tilde{F_{\beta}}(0,0)h(w)=\mu_{\beta}\textnormal{Im}[h'(w)]+\textnormal{Re}\left[\overline{K(h)(w)}w\right].
\end{align*}
Note now that $K:\mathscr{C}^{2-\beta}(\T)\rightarrow \mathscr{C}^{1-\beta}(\T)$ is a compact operator since it is smoothing (see for instance Lemma \ref{Lem-pottheory}). On the other hand $h\in X_{2-\beta}\mapsto h'\in Y_{1-\beta}$ is an isomorphism and thus it is a Fredholm operator of zero index. Hence, since compact perturbations of Fredholm operators remain Fredholm with same index, we have that $\partial_f \tilde{F_{\beta}}(0,0)$ is Fredholm of zero index. As a consequence, it is enough to check that is has a trivial kernel to show that it is an isomorphism.

Using \eqref{int-1-g}--\eqref{int-2-g} one achieves
\begin{align*}
\partial_f \tilde{F_{\beta}}(0,0)h(w)=\sum_{n\geq 1}a_n \sin((n+1)\theta)\left\{\mu_{\beta}+2(\gamma_{1,n}-\gamma_{2,n})-\frac{2nC_\beta \Gamma(1-\beta)}{\Gamma^2(1-\beta/2)}\frac{\left(\frac{\beta}{2}\right)_n}{\left(1-\frac{\beta}{2}\right)_n}\right\}.
\end{align*}
Finally, with the help of the computations done in \cite[Pages 726--728]{H-M}, one gets
\begin{align*}
\partial_f \tilde{F_{\beta}}(0,0)h(w)=\sum_{n\geq 1}\frac{C_\beta \beta\Gamma(1-\beta)}{2\Gamma^2(1-\beta/2)} a_n\gamma_n \sin((n+1)\theta),
\end{align*}
with
$$
\gamma_n:=\frac{2(1+n)}{1-\beta/2}-\frac{(1+\beta/2)_n}{(1-\beta/2)_n}-\frac{(1+\beta/2)_{n+1}}{(1-\beta/2)_{n+1}}.
$$
In order to have a trivial kernel, one needs that $\gamma_n\neq 0$ for any $n\geq 1$. Indeed by \cite[Page 728]{H-M} one has
$$
\gamma_n>\frac{2(1+n)}{1-\beta/2}-\frac{1+\beta/2}{2-\beta/2}\frac{n+\beta/2}{1-\beta/2}-\frac{1+\beta/2}{2-\beta/2}\frac{n+1+\beta/2}{1-\beta/2}>C_0 n,
$$
for some constant $C_0>0$. Hence, the kernel is trivial and we can ensure that $\partial_f \tilde{F_{\beta}}(0,0)$ is an isomorphism. 
\end{proof}



\begin{thebibliography}{99}
%
%
%
\bibitem{Ao}{W. Ao, J. D\'avila, M. Del Pino, M. Musso, J. Wei, }{\it Travelling and rotating solutions to the generalized inviscid surface quasi--geostrophic equation. } arXiv:2008.12911.

\bibitem{B-C} A. L. Bertozzi, P. Constantin, {\it Global regularity for vortex patches.} Comm. Math. Phys.
{\bf 152}(1) (1993), 19--28.

\bibitem{Boatto}{S. Boatto, C. Sim\'o, }{\it Thomson's Heptagon: A case of bifurcation at infinity. } Physica D {\bf 237} (2008), 2051--2055.
%
\bibitem{Burbea} { J. Burbea,} {\it Motions of vortex patches.} Lett. Math. Phys. {\bf 6} (1982), 1--16.
%

\bibitem{Calleja}{R. C. Calleja, E. J. Doedel, C. Garc\'ia--azpeitia, }{\it Choreographies in the n--vortex problem. } Regul. Chaot. Dyn {\bf 23} (2018), 595--612.

 \bibitem{Cas0-Cor0-Gom} {A. Castro, D. C\'ordoba, J. G\'omez-Serrano, }{\it Existence and regularity of rotating global solutions for the generalized surface quasi-geostrophic equations.} Duke Math. J. {\bf 165}(5) (2016), 935--984.
 
\bibitem{Cas-Cor-Gom} {A. Castro, D. C\'ordoba, J. G\'omez-Serrano, }{\it  Uniformly rotating analytic global patch solutions for active scalars}. J. Ann. PDE {\bf 2}(1) (2016),  Art. 1, 34.

\bibitem{CastroCordobaGomezSerrano} {A. Castro, D. C\'ordoba, J. G\'omez-Serrano, } {{\it Uniformly rotating smooth solutions for the incompressible 2D Euler equations.}} Arch. Ration. Mech. Anal. {\bf 231}(2) (2019), 719--785.

\bibitem{C-C-GS-2} A. Castro, D. C\'ordoba,  J. G\'omez-Serrano, {\it Global smooth solutions for the inviscid SQG equation,} arXiv:1603.03325, 2016.
\bibitem{Chemin} {J.-Y. Chemin,} {{\it Persistance de structures g\'eometriques dans les  fluides incompressibles bidimensionnels. }} Ann. Sci. Ec. Norm. Sup. {\bf 26} (1993), 1--26.
%
%
\bibitem{Constantin-Majda-Tabak}{P. Constantin, A. J. Majda, E. Tabak, }{\it Formation of strong fronts in the 2-D quasigeostrophic thermal
active scalar.} Nonlinearity {\bf 7}(6) (1994), 1495--1533.
%

\bibitem{CordobaCordoba}{A. C\'ordoba, D. C\'ordoba, F. Gancedo, }{\it Uniqueness for SQG patch solutions. } Trans. Amer. Math. Soc. Ser. B {\bf 5} (2018), 1--31.

\bibitem{Cordoba}{D. C\'ordoba, M.A. Fontelos, A.M. Mancho, J.L. Rodrigo, }{\it Evidence of singularities for a family of
contour dynamics equations.} Proc. Natl. Acad. Sci. USA {\bf 102}(17) (2005), 5949--5952.
%





%

\bibitem{DavilaPinoMussoWei}{J. Davila, M. Del Pino, M. Musso, and J. Wei, }{\it Gluing Methods for Vortex Dynamics
in Euler Flows.} Arch. Ration. Mech. Anal. {\bf 235} (2020), no. 3, 1467--1530.
%
 \bibitem{D-H-R} {D. G. Dritschel, T. Hmidi, C. Renault, } {\it
Imperfect bifurcation for the quasi-geostrophic shallow-water equations}, arXiv:1801.02092, 2018.
%
\bibitem{DelaHoz-Hassainia-Hmidi}{F. De la Hoz, Z. Hassainia, T. Hmidi}{\it Doubly Connected V-States for the Generalized Surface Quasi-Geostrophic Equations.} Arch. Ration. Mech. Anal {\bf 220} (2016), 1209--1281. 

\bibitem{DelaHozHmidiMateuVerdera} {F. De la Hoz, T. Hmidi, J. Mateu, J. Verdera, } {\it Doubly connected V-states for the planar Euler equations. } SIAM J. Math. Anal. {\bf 48} (2016), 1892--1928.
%
\bibitem{Deem-Zabusky} {G. S. Deem, N. J. Zabusky, } {\it Vortex waves: Stationary ``V-states'', Interactions, Recurrence, and Breaking.} Phys. Rev. Lett. {\bf 40} (1978), 859--862.
%

\bibitem{Gancedo}{F. Gancedo. }{\it Existence for the $\alpha$--patch model and the QG sharp front in Sobolev spaces. } Adv. Math., {\bf 217}(6) (2008), 2569--2598.


\bibitem{G-KVS}{C. Garc\'ia, }{\it K\'arm\'an Vortex Street in incompressible fluid models}. Nonlinearity {\bf 33}(4) (2020), 1625--1676.

\bibitem{GHM}{C. Garc\'ia, T. Hmidi, J. Mateu, }{\it Time periodic solutions for 3D quasi--geostrophic model. } arXiv:2004.01644, 2020.

\bibitem{GHS}{C. Garc\'ia, T. Hmidi, J. Soler, }{\it Non uniform rotating vortices and periodic orbits for the two--dimensional Euler equations. } Arch. Ration. Mech. Anal {\bf 238} (2020), 929--1086.


\bibitem{GPSY}{ J.  G\'omez-Serrano, J. Park, J. Shi, Y.Yao} {\it Symmetry in stationary and uniformly-rotating solutions of active scalar equations}, arXiv:1908.01722, 2019. 

  \bibitem{Hassa-Hmi}  Z. Hassainia, T. Hmidi, {\it  On the V-states for the generalized quasi-geostrophic equations.} Comm. Math. Phys. {\bf 337}(1) (2015), 321--377.
%
\bibitem{HMW} Z. Hassainia, N. Masmoudi, M. H. Wheeler, {\it Global bifurcation of rotating vortex patches},  	arXiv:1712.03085, 2017.
%

\bibitem{Helms}{L. L. Helms, }{\it Potential theory, } Springer-Verlag London, 2009.


 \bibitem{Hmiditrivial} { T. Hmidi,} {\it On the trivial solutions for the vortex patch rotating model,} J. Evol. Equ. {\bf 15} (2015), 801--816.
%


\bibitem{HmidiMateu}{T. Hmidi, J. Mateu, } {\it Bifurcation of rotating patches from Kirchhoff vortices.} Discret. Contin. Dyn. Syst. {\bf 36} (2016), 5401--5422.

\bibitem{H-M} T. Hmidi, J. Mateu,  {\it Existence of corotating and counter-rotating vortex pairs for active
scalar equations.} Comm. Math. Phys. {\bf 350}(2) (2017), 699--747.
%
\bibitem{HmidiMateuVerdera} {T. Hmidi, J. Mateu, J. Verdera, } {\it Boundary regularity of rotating vortex patches. } Arch. Ration. Mech. Anal {\bf 209} (2013), 171--208. 
%

\bibitem{Kirchhoff}{G. R. Kirchhoff, }{\it  Vorlesungenber mathematische Physik. Mechanik.} Teubner, Leipzig, 1876.

\bibitem{Kiselev}{A. Kiselev, Y. Yao, A. Zlatos, } {\it Local regularity for the modified SQG patch equation. }Comm. Pure Appl. Math., {\bf 70}(7) (2017), 1253--1315.
%
\bibitem{Kress}{R. Kress, } {\it Linear Integral Equations. } Springer--New York--Heidelberg Dordrecht--London, 2014.

\bibitem{KurakinYudovich}{L. G. Kurakin, V. I. Yudovich, }{\it The stability of stationary rotation of a regular vortex polygon. } Chaos: An Interdisciplinary Journal of Nonlinear Sciences {\bf 12} (2020), 574--595.

\bibitem{LiebLoss}{E. Lieb, M. Loss, } {\it Analysis. } American Mathematical Society, 1997.



\bibitem{Marchioro}{C. Marchioro, M. Pulvirenti, }{\it Euler evolution for singular initial data and vortex theory.} Comm.
Math. Phys. {\bf 91}(4) (1983), 563--572.


\bibitem{Newton}{P. K. Newton, }{\it The N-Vortex Problem.} Analytical Techniques, Springer, New York, 2001.
%


\bibitem{Rodrigo}{J. L. Rodrigo, }{\it On the evolution of sharp fronts for the quasi-geostrophic equation. } Comm. Pure Appl. Math.,
{\bf 58}(6) (2005), 821--866.
%
 

\bibitem{Serf} P. Serfati, {\it Une preuve directe d'existence globale des vortex patches 2D.}
C. R. Acad. Sci. Paris S\'er. I Math. {\bf 318}(6) (1994), 515--518.

\bibitem{Smets}{D. Smets. J. Van Schaftingen, }{\it Desingularization of vortices for the Euler equation.} Arch.
Ration. Mech. Anal. {\bf 198}(3) (2010), 869--925.

\bibitem{Yudovich}{Y. Yudovich, }{\it Nonstationary flow of an ideal incompressible liquid. } Zh. Vych. Mat. {\bf 3} (1963), 1032--1066.
%
\end{thebibliography}
\end{document}